\documentclass[11pt]{article}
\usepackage{fullpage}
\usepackage{amssymb, amsmath}
\usepackage{amsthm}
\usepackage{paralist}
\usepackage[colorlinks=true]{hyperref}
\usepackage{mathrsfs}
\hypersetup{urlcolor=blue, citecolor=red}
\usepackage{enumerate}
\usepackage{enumitem}

\newtheorem{theorem}{Theorem}[section]
\newtheorem{lemma}[theorem]{Lemma}
\newtheorem{proposition}[theorem]{Proposition}
\theoremstyle{definition}
\newtheorem{definition}[theorem]{Definition}
\newtheorem{remark}[theorem]{Remark}
\newtheorem{assumption}[theorem]{Assumption}
\newtheorem*{notation}{Notation}

\newcommand{\norm}[1]{\left\lVert#1\right\rVert}
\newcommand{\R}{\mathbb{R}}
\newcommand{\N}{\mathbb{N}}
\newcommand{\E}{\mathcal{E}}
\newcommand{\rcal}{\mathcal{R}}

\newcommand{\e}{_{\varepsilon}}
\newcommand{\h}{_{\mathrm{hom}}}
\newcommand{\p}{_{\mathrm{pot}}}
\newcommand{\harpoon}{\rightharpoonup}
\newcommand{\pinv}{P_{{\mathrm{inv}}}} 
\newcommand{\expect}[1]{\left\langle {#1} \right\rangle}
\newcommand{\unf}{\mathcal{T}_{\varepsilon}} 
\newcommand{\brac}[1]{\left({#1}\right) } 
\newcommand{\cb}[1]{\left\lbrace {#1} \right\rbrace}
\newcommand{\wt}{\overset{2}{\rightharpoonup}} 
\newcommand{\st}{\overset{2}{\rightarrow}}

\usepackage{authblk}
\title{Stochastic homogenization of $\Lambda$-convex gradient flows}
\author[1]{Martin Heida\thanks{martin.heida@wias-berlin.de}} 
\author[2]{Stefan Neukamm\thanks{stefan.neukamm@tu-dresden.de}}
\author[2]{Mario Varga\thanks{mario.varga@tu-dresden.de}}
\affil[1]{Weierstrass Institute for Applied Analysis and Stochastics, Berlin}
\affil[2]{Faculty of Mathematics, Technische Universit\"at Dresden}
\begin{document}
\maketitle

\centerline{\textit{This paper is dedicated to Alexander Mielke on the occasion of his 60th birthday.}}
\begin{abstract}
In this paper we present a stochastic homogenization result for a  class of Hilbert space evolutionary gradient systems driven by a quadratic dissipation potential and a $\Lambda$-convex energy functional featuring random and rapidly oscillating coefficients. Specific examples included in the result are Allen-Cahn type equations and evolutionary equations driven by the $p$-Laplace operator with $p\in (1,\infty)$. The homogenization procedure we apply is based on a stochastic two-scale convergence approach. In particular, we define a stochastic unfolding operator which can be considered as a random counterpart of the well-established notion of periodic unfolding. The stochastic unfolding procedure grants a very convenient method for homogenization problems defined in terms of ($\Lambda$-)convex functionals.

\noindent
{\bf Keywords:} Stochastic homogenization, stochastic unfolding, two-scale convergence, gradient system.
\end{abstract}
\section{Introduction}
Homogenization theory deals with the derivation of effective, macroscopic models for problems that involve two or more length (or time) scales. In stochastic homogenization the considered models are described in terms of coefficient fields that are randomly varying on a small scale, say $0<\varepsilon\ll 1$. A typical situation involves stationary random coefficient fields of the form $\R^d\ni x\mapsto a(\omega,\frac{x}{\varepsilon})=a_0(\tau_{\frac{x}{\varepsilon}}\omega)$ where $\omega\in \Omega$ stands for a ``random configuration'' and $a_0$ is defined on a probability space $(\Omega,\mathcal{F},P)$ that is equipped with a measure preserving action $\tau_{x}:\Omega\to \Omega$, see Section \ref{Intro} for the precise description of random coefficients.

In this paper we consider stochastic homogenization of gradient flows defined in terms of two integral functionals with random and rapidly-oscillating integrands---a quadratic \textit{dissipation} functional $\rcal\e: Y\to \R$ and a $\Lambda$-convex \textit{energy} functional $\E\e: Y\to \R\cup \cb{\infty}$. In particular, these functionals are defined on a state space $Y=L^2(\Omega\times Q)$ (the dual space is denoted by $Y^*$), where $Q\subset \R^d$ is open and bounded, and they admit the form
\begin{align*}
\rcal_{\varepsilon}(\dot{y}) & = \frac{1}{2}\int_{\Omega}\int_{Q}r(\tau_{\frac{x}{\varepsilon}}\omega,x)|\dot{y}(\omega,x)|^2 dx dP(\omega),\\
\mathcal{E}_{\varepsilon}(y) & = \int_{\Omega}\int_{Q} V(\tau_{\frac{x}{\varepsilon}}\omega,x, \nabla y(\omega,x))+ f(\tau_{\frac{x}{\varepsilon}}\omega,x,y(\omega,x))dx dP(\omega).
\end{align*}
Besides usual measurability statements, the main assumptions for $V(\omega,x,\cdot)$ are convexity and $p$-growth conditions with $p \in (1,\infty)$, and we assume that $f(\omega,x,\cdot)$ has $\theta$-growth with $\theta\in [2,\infty)$ and it is $\lambda$-convex, i.e., there exists $\lambda \in \R$ such that $f(\omega,x,\cdot)-\frac{\lambda}{2}|\cdot|^2$ is convex. The latter implies that $\E\e(\cdot)-\Lambda \rcal\e(\cdot)$ is convex for suitable $\Lambda\in \R$, i.e., $\E\e$ is $\Lambda$-convex w.r.t. $\rcal\e$. For the precise definitions and assumptions, see Section \ref{Intro}.

The evolution of the gradient flow is described by a state variable $y\in H^1(0,T;Y)$ and it is determined by the following differential inclusion
\begin{equation}\label{eq:145}
0 \in D\rcal_{\varepsilon}(\dot{y}(t))+ \partial_{F}\mathcal{E}_{\varepsilon}(y(t)) \quad \text{for a.e. }t\in (0,T), \quad y(0)=y^0\in Y. 
\end{equation}
Above, $\partial_F \E\e: Y \to 2^{Y^*}$ denotes the Frech{\'e}t subdifferential (see \cite{kruger2003frechet}), which is, in the specific case of a $\Lambda$-convex energy $\E\e$, given by: $\xi \in \partial_F \E\e(y)$ if
\begin{equation*}
\E\e(y)\leq \E\e(\widetilde{y})+\expect{\xi,y-\widetilde{y}}_{Y^*,Y}-\Lambda \rcal\e(\widetilde{y}-y) \quad \text{for all }\widetilde{y}\in Y.
\end{equation*} 
In this regard, the differential inclusion from \eqref{eq:145} is equivalent to the evolutionary variational inequality (\textit{EVI})
\begin{equation}\label{eq:569:p}
\expect{D\rcal\e(\dot{y}(t)),y(t)-\widetilde{y}}_{Y^*,Y}\leq \E\e(\widetilde{y})-\E\e(y(t))-\Lambda \rcal\e(y(t)-\widetilde{y}), \tag{$\mathrm{EVI}$}
\end{equation}
$\text{for all }\widetilde{y}\in Y$. We refer to the textbooks \cite{brezis1973ope,zeidler2013nonlinear,roubicek2013nonlinear,ambrosio2008gradient} for a general and detailed theory of gradient flows. In the simple case  $V(\omega,x,F)=A(\omega,x)F\cdot F$ and $f(\omega,x,\alpha)=\alpha^4 -\alpha^2$, \eqref{eq:145} corresponds to the weak formulation of an Allen-Cahn equation. Also, in the case that $V(\omega,x,F)=a(\omega,x)|F|^p$ with $p\in (1,\infty)$, the evolution is driven by the $p$-Laplace operator with oscillatory coefficients.

In the limit $\varepsilon\to 0$, we derive an effective gradient flow given in terms of a state space $Y_0= L^2_{\mathrm{inv}}(\Omega)\otimes L^2(Q)$ and homogenized functionals $\rcal\h: Y_0\to \R$, $\E\h: Y_0\to \R\cup \cb{\infty}$, see Section \ref{Intro} for the specific definitions. In particular, we obtain the following \textit{well-prepared E-convergence} statement for the limit $\varepsilon\to 0$: 
\begin{align*}
& \text{If }y\e(0)\to y(0) \quad \text{strongly in }Y, \quad \E\e(y\e(0)) \to \E\h(y(0)),\\
& \text{then for all }t\in [0,T], \quad y\e(t)\to y(t) \quad \text{strongly in }Y, \quad \E\e(y\e(t))\to \E\h(y(t)),
\end{align*}
where $y\e$ and $y$ denote the unique solutions to the gradient flows given in terms of $\brac{Y,\E\e,\rcal\e}$ and $\brac{Y_0,\E\h,\rcal\h}$, respectively (see Theorem \ref{s3_thm_5}).

The proof of this homogenization result relies on a general approach for asymptotic analysis of gradient flows and on the stochastic unfolding procedure, which we briefly explain in the following:

\textbf{General approach.} In the last decades, a number of general strategies for asymptotic analysis of sequences of abstract gradient systems were developed, we refer to \cite{mielke2016evolutionary} for a comprehensive overview. In particular, an early contribution in this field is obtained in \cite{attouch1978convergence,attouch1984variational}, where gradient flows on an abstract Hilbert space with fixed dissipation potential $\rcal\e=\rcal$ and convex energy functionals $\E\e$ are considered. In this setting, e.g., Mosco convergence $\E\e\overset{M}{\to}\E_0$ is sufficient to conclude well-prepared E-convergence. Novel strategies have been developed in  \cite{sandier2004gamma,serfaty2011gamma} and \cite{mielke2013nonsmooth}, which allow the treatment of very general problems with varying (nonquadratic, convex) dissipation potentials $\rcal\e$ and possibly nonconvex energy functionals $\E\e$. They are based on \textit{De Giorgi's $\brac{\rcal,\rcal^*}$} formulation (see, e.g., \cite[Introduction]{mielke2016evolutionary}). Also, using an integrated version of the \eqref{eq:569:p} formulation, in \cite{daneri2010lecture} a method for sequences with $\Lambda$-convex energies is proposed (see also \cite{mielke2014deriving}). In \cite{stefanelli2008brezis}, the \textit{Brezis-Ekeland-Nayroles principle} is utilized for the development of a procedure for E-convergence for convex dissipation and energy functionals. 

Many approaches for proving  $E$-convergence for problems  with \textit{nonconvex} energy functionals rely on the relative compactness in $Y$ of the energy ``sublevels'' $\cb{y\in Y:\E\e(y)\leq c, \; \forall \varepsilon}$ (or a similar strong-type compactness property). In our specific problem (which involves a nonconvex, $\Lambda$-convex energy functional) we only have compactness in weak topologies at our disposal. The lack of compactness in a strong topology is due to two reasons. The first reason comes from the fact that we consider convergence in the $L^2$-probability space: While in the deterministic periodic case (i.e., when $x\mapsto\tau_x\omega$ is periodic  almost surely), the compact embedding $H^{1}(Q)\subset \subset L^2(Q)$ yields strong compactness of the energy sublevels if $p=2$, in the general stochastic setting, the embedding of $L^2(\Omega)\otimes H^1(Q)$ into $L^2(\Omega\times Q)$ is not compact. The second reason is a possible mismatch between the growth of  $f$ and the growth control via $V$: If $p<2$ and $d$ is large, then even in the deterministic periodic case we are not able to obtain apriori strong $L^2$-type compactness. For this reason, we consider a modified approach that we briefly describe in the following and we refer to Sections~\ref{Intro}~and~\ref{Section_3.2} for details.

We define a new time-dependent energy functional $\widetilde{\E}\e:[0,T]\times Y\to \R\cup \cb{\infty}$,
\begin{equation*}
 \widetilde{\E}\e(t,u)=e^{2\Lambda t}\E\e(e^{-\Lambda t}u)-\Lambda \rcal\e(u),
\end{equation*}
for which $\widetilde{\E}\e(t,\cdot)$ is convex. If $y\e$ satisfies \eqref{eq:569:p} a.e., then using the Fenchel equivalence the new variable $u\e(t):=e^{\Lambda t}y\e(t)$ fulfills (cf. Lemma \ref{lem:709:p})
\begin{equation}\label{eq:190}
\expect{D\rcal\e(\dot{u}\e(t)),u\e(t)}_{Y^*,Y}+\widetilde{\E}\e(t,u\e(t))+\widetilde{\E}^*\e(t,-D\rcal\e( \dot{u}\e(t)))=0,
\end{equation}
where $\widetilde{\E}\e^*(t,\cdot)$ denotes the convex conjugate of $\widetilde{\E}\e(t,\cdot)$.
Using the chain rule and the quadratic structure of $\rcal\e$ in form of $(D\mathcal R_{\varepsilon})^*=D\mathcal R_{\varepsilon}$, we have $\frac{d}{dt}\rcal\e(u\e(t))=\expect{D\rcal\e(u\e(t)),\dot{u}\e(t)}_{Y^*,Y}=\expect{D\rcal\e(\dot{u}\e(t)),u\e(t)}_{Y^*,Y}$. Hence, an integration of \eqref{eq:190} over $(0,T)$ yields
\begin{equation}\label{eq:509:p}
\rcal\e(u\e(T))+\int_{0}^T \widetilde{\E}\e(t,u\e(t))+\widetilde{\E}^*\e(t,-D\rcal\e(\dot{u}\e(t)))dt = \rcal\e(u\e(0)).
\end{equation}
This formulation is equivalent to \eqref{eq:569:p} and it is convenient for passing to the limit $\varepsilon\to 0$ by only using  weak convergence of the solution $y\e$ (resp. $u\e$). In fact, \eqref{eq:509:p} is the analogue of the formulation used in the general convex case in \cite{attouch1978convergence,attouch1984variational} with the difference that in our case the energy functionals are time dependent and that the dissipation functionals feature oscillations on scale $\varepsilon$. 

\textbf{Stochastic unfolding.} In order to conduct the limit passage $\varepsilon\to 0$ in \eqref{eq:509:p}, we are required to treat objects with random and rapidly oscillating coefficients. For this task, we introduce the \textit{stochastic unfolding method} that allows a straightforward analysis and it presents a random counterpart of the well-established periodic unfolding method. 

The notion of periodic two-scale convergence \cite{nguetseng1989general,allaire1992homogenization} (see also \cite{lukkassen2002two}) and the periodic unfolding procedure \cite{cioranescu2002periodic} (see also \cite{cioranescu2008periodic,Visintin2006,mielke2007two}) are prominent and useful tools in multiscale modeling and homogenization suited for problems involving periodic coefficients. We refer to some of the many problems treated using these methods \cite{lukkassen2002two,cioranescu2004homogenization,griso2004error,
mielke2007two,neukamm2010homogenization,mielke2014two,
liero2015homogenization,hanke2017phase}. 
In the stochastic setting, the notion of two-scale convergence is generalized in \cite{bourgeat1994stochastic} (see also \cite{andrews1998stochastic,sango2011stochastic}) and in \cite{zhikov2006homogenization} (see also \cite{mathieu2007quenched,faggionato2008random,heida2011extension}). Yet, as far as we know, the concept of unfolding has not been investigated earlier in the stochastic case. 

We extend the idea of the periodic unfolding procedure to the stochastic case. Namely, we introduce a linear isometric operator, the \textit{stochastic unfolding operator}, that enjoys many similarities to the periodic unfolding operator. Also, similarly as in the periodic case, stochastic two-scale convergence in the mean from \cite{bourgeat1994stochastic} might be equivalently characterized as weak convergence of the unfolded sequence. In this respect, we develop a general procedure for stochastic homogenization problems, see also \cite{varga2019stochastic} for a detailed analysis of this method, and \cite{nvw2019} for an extension to abstract, linear evolution systems in an operator theoretic framework.
Stochastic unfolding has first been introduced by the second and third author in a discrete version in \cite{neukamm2017stochastic} where the discrete-to-continuum limit of a rate-independent evolution is analyzed.
\medskip

\noindent\textbf{Related results.} In the periodic setting homogenization results of this type are obtained for quasilinear parabolic equations, e.g., in \cite{neuss2007effective,woukeng2010periodic,fatima2012unfolding} (via two-scale convergence and unfolding), for reaction-diffusion systems with different diffusion length scales in \cite{mielke2014two} (via unfolding), for Cahn-Hilliard type gradient flows in \cite{liero2015homogenization} (via unfolding). In the stochastic case, parabolic type equations are treated in \cite{zhikov1982averaging,delarue2009stochastic,heida2012stochastic,efendiev2005homogenization}. However, the approach we consider is different, it relies on the more general gradient flow formulation and we do not rely on differentiability of the integrands $V$ and $f$ and on continuity assumptions on their derivatives.

\medskip  
\noindent\textbf{Structure of the paper.} In Section \ref{Intro} we present the main stochastic homogenization result of this paper. Section \ref{S_Stoch} is dedicated to the introduction of the stochastic unfolding procedure. In Section \ref{Section_3.2} we present the proof of the main Theorem \ref{s3_thm_5}.

\begin{notation} $(\Omega,\mathcal{F},P)$ denotes a complete and separable probability space, the corresponding mathematical expectation is denoted by $\expect{\cdot}=\int_{\Omega} \cdot dP(\omega)$. For $Q\subset \R^d$ open, we denote by $\mathcal{L}(Q)$ the Lebesgue $\sigma$-algebra. For a Banach space $X$, its dual space is denoted by $X^*$ and the Borel $\sigma$-algebra on $X$ is given by $\mathcal{B}(X)$. For $p\in (1,\infty)$, $L^p(\Omega)$ and $L^p(Q)$ are the usual Banach spaces of $p$-integrable functions defined on $(\Omega,\mathcal F,P)$ and $Q$, respectively. We introduce function spaces for functions defined on $\Omega\times Q$ as follows: For closed subspaces $X\subset L^p(\Omega)$ and $Z\subset L^p(Q)$, we denote by $X\otimes Z$ the closure of $$X\overset{a}{\otimes}Z:=\cb{\sum_{i=1}^{n}\varphi_i \eta_i:  \varphi_i \in X, \eta_i\in Z, n\in \mathbb{N}}$$ in $L^p(\Omega\times Q)$. Note that in the case $X=L^p(\Omega)$ and $Z=L^p(Q)$, we have $X\otimes Z = L^p(\Omega\times Q)$. Up to isometric isomorphisms, we may identify $L^p(\Omega\times Q)$ with the Bochner spaces $L^p(\Omega;L^p(Q))$ and $L^p(Q;L^p(\Omega))$. Slightly abusing the notation, for closed subspaces $X\subset L^p(\Omega)$ and $Z\subset W^{1,p}(Q)$, we denote by $X \otimes Z$ the closure of
$$X\overset{a}{\otimes}Z:=\cb{\sum_{i=1}^{n}\varphi_i \eta_i:  \varphi_i \in X, \eta_i\in Z, n\in \mathbb{N}}$$
in $L^p(\Omega;W^{1,p}(Q))$. In this regard, we may identify $u\in L^p(\Omega)\otimes W^{1,p}(Q)$ with the pair $(u,\nabla u)\in L^p(\Omega\times Q)^{1+d}$. We mostly focus on the space $L^p(\Omega\times Q)$ and the above notation is convenient for keeping track of its various subspaces.
\end{notation}
\section{Homogenization of gradient flows}\label{Intro} First, we briefly recall the standard functional analytic setting for stochastic homogenization introduced by Papanicolaou and Varadhan in \cite{Papanicolaou1979} (see also \cite{jikov2012homogenization}). In the second part of this section we present the main homogenization result.
\begin{assumption} \label{Assumption_2_1}
Let $\brac{\Omega,\mathcal{F},P}$ be a complete and separable probability space. Let $\tau=\cb{\tau_x}_{x\in \R^{d}}$ denote a group of invertible measurable mappings $\tau_x:\Omega\to \Omega$ such that:
\begin{enumerate}[label=(\roman*)]
\item (Group property). $\tau_0=Id$ and $\tau_{x+y}=\tau_x\circ \tau_y$ for all $x,y\in \R^{d}$.
\item (Measure preservation). $P(\tau_x E)=P(E)$ for all $E\in \mathcal{F}$ and $x\in \R^{d}$.
\item (Measurability). $(\omega,x)\mapsto \tau_{x}\omega$ is $\brac{\mathcal{F}\otimes \mathcal{L}(\R^d),\mathcal{F}}$-measurable.
\end{enumerate}
\end{assumption}
Throughout the paper we assume that $(\Omega,\mathcal F,P,\tau)$ satisfies Assumption \ref{Assumption_2_1}.
The separability assumption on the measure space implies that $L^p(\Omega)$ is separable. We say that $(\Omega,\mathcal F,P,\tau)$ is \textit{ergodic} ($\expect{\cdot}$ is ergodic), if
\begin{align*}
  \text{ every shift invariant }E\in \mathcal{F} \text{ (i.e.,~}\tau_x E=E \text{ for all }x\in \R^d)\text{ satisfies } P(E)\in \cb{0,1}. 
\end{align*}

We introduce two auxiliary subspaces of $L^p(\Omega)$ that are important for the homogenization procedure. We consider the group of isometric operators $\cb{U_x}_{x\in \R^{d}}$,  $U_x:L^p(\Omega)\to L^p(\Omega)$ defined by $U_x \varphi(\omega)=\varphi(\tau_{x}\omega)$. This group is strongly continuous (see \cite[Section 7.1]{jikov2012homogenization}). For $i=1,...,d$, we consider the one-parameter group of operators $\cb{U_{h e_i}}_{h\in \R}$ ($\cb{e_i}$ being the usual basis of $\R^d$) and its infinitesimal generator $D_i:\mathcal{D}_i\subset L^p(\Omega)\rightarrow L^p(\Omega)$,
\begin{equation*}
D_i \varphi=\lim_{h\rightarrow 0} \frac{U_{he_i}\varphi-\varphi}{h},
\end{equation*}
which we refer to as the \textit{stochastic} derivative. $D_i$ is a linear and closed operator and its domain $\mathcal{D}_i$ is dense in $L^p(\Omega)$. We set $W^{1,p}(\Omega)=\cap_{i=1}^{d}\mathcal{D}_i$ and define for $\varphi\in W^{1,p}(\Omega)$ the \textit{stochastic gradient} as $D\varphi=(D_1 \varphi,...,D_d \varphi)$. In this manner, we obtain  a linear, closed and densely defined operator $D:W^{1,p}(\Omega)\rightarrow L^p(\Omega)^d$, and we denote by
\begin{equation*}
L^p_{\mathrm{pot}}(\Omega):=\overline{\mathrm{ran}(D)}\subset L^p(\Omega)^d
\end{equation*}
the closure of the range of $D$ in $L^p(\Omega)^d$. We denote the adjoint of $D$ by $D^*:\mathcal{D}^*\subset{L^q(\Omega)^d}\rightarrow L^q(\Omega)$ which is a linear, closed and densely defined operator, $\mathcal{D}^*$ denotes the domain of $D^*$ and $q=\frac{p}{p-1}$. 
Note that $W^{1,q}(\Omega)^d\subset \mathcal{D}^*$ and for all $\varphi\in W^{1,p}(\Omega)$ and $\psi\in W^{1,q}(\Omega)$ we have the integration by parts formula, $i=1,...,d$,
\begin{equation*}
\expect{ \psi D_i \varphi}=-\expect{\varphi D_i \psi},
\end{equation*}
and thus $D^*\psi=-\sum_{i=1}^d D_i \psi_i$ for $\psi\in W^{1,q}(\Omega)^d$. We define the subspace of \textit{shift-invariant functions} in $L^p(\Omega)$ by 
\begin{equation*}
L^p_{{\mathrm{inv}}}(\Omega)=\cb{\varphi\in L^p(\Omega):U_x \varphi=\varphi \quad \text{for all }x \in \R^{d}},
\end{equation*}
and denote by $ P_{\mathrm{inv}}: L^p(\Omega) \to L^p_{\mathrm{inv}}(\Omega)$ the conditional expectation with respect to the $\sigma$-algebra of shift invariant sets $\cb{ E \in \mathcal{F} : \tau_x E = E \text{ for all } x\in \R^d}$. $P_{\mathrm{inv}}$ is a contractive projection and for $p=2$ it coincides with the orthogonal projection onto $L^2_{\mathrm{inv}}(\Omega)$. Also, if $\expect{\cdot}$ is ergodic, then it holds $L^p_{\mathrm{inv}}(\Omega)\simeq \R$ and $P_{\mathrm{inv}}\varphi = \expect{\varphi}$.
\smallskip

\noindent \textbf{Heterogeneous system.} 
Let $Q\subset \R^d$ be open and bounded. Let $p\in (1,\infty)$ and $\theta\in [2,\infty)$.
The system that we consider is defined on a state space

$$Y = L^2(\Omega \times Q).$$
The dissipation functional is given by $\mathcal{R}_{\varepsilon}: Y\to [0,\infty)$,
\begin{equation*}
\mathcal{R}_{\varepsilon}(\dot{y})=\frac{1}{2}\expect{\int_{Q}r(\tau_{\frac{x}{\varepsilon}}\omega,x)|\dot{y}(\omega,x)|^2 dx}.
\end{equation*}
The energy functional $\E_{\varepsilon}:Y\to \R\cup \cb{\infty}$ is defined as
\begin{equation*}
\E_{\varepsilon}(y)=\expect{\int_{Q} V(\tau_{\frac{x}{\varepsilon}}\omega,x, \nabla y(\omega,x))+f(\tau_{\frac{x}{\varepsilon}}\omega,x,y(\omega,x))dx},
\end{equation*}
for $y \in (L^p(\Omega)\otimes W^{1,p}_0(Q))\cap L^{\theta}(\Omega \times Q)=:\mathrm{dom}(\mathcal{E}\e)$ and $\E_{\varepsilon} = \infty$ otherwise. Above, $r:\Omega \times Q \to \R$, $V: \Omega\times Q \times \R^d\to \R$ and $f:\Omega\times Q \times \R \to \R$ and we consider the following assumptions: There exists $c>0$ such that:
\begin{enumerate}[label=(A\arabic*)]
\item \label{C1:ass:a} $r$ is $\mathcal{F}\otimes \mathcal{L}(Q)$-measurable and for a.e. $(\omega,x)\in \Omega\times Q$, we have $\frac{1}{c}\leq r(\omega,x)\leq c$.
\item \label{C2:ass:a} $V(\cdot,\cdot, F)$ is $\mathcal{F}\otimes \mathcal{L}(Q)$-measurable for all $F\in \R^{d}$, $V(\omega, x, \cdot)$ is convex for a.e. $(\omega,x)\in \Omega\times Q$ and
\begin{equation}\label{eq:271}
\frac{1}{c}|F|^p-c\leq V(\omega, x, F) \leq c(|F|^p+1)
\end{equation}
for a.e. $(\omega,x) \in \Omega\times Q$ and all $F\in \R^{d}$.
\item \label{C3:ass:a} $f(\cdot,\cdot,\alpha)$ is $\mathcal{F}\otimes \mathcal{L}(Q)$-measurable for all $\alpha \in \R$. There exists $\lambda \in \R$ such that for a.e. $(\omega,x) \in \Omega\times Q$ 
\begin{align}\label{eq:276}
& f(\omega,x, \cdot) \text{ is }\lambda\text{-convex, i.e., } \quad \alpha \mapsto f(\omega,x, \alpha)-\frac{\lambda}{2} |\alpha|^2 \text{ is convex},\nonumber \\
&\frac{1}{c}|\alpha|^{\theta} - c \leq f(\omega,x,\alpha)\leq c(|\alpha|^{\theta}+1) \quad \text{for all }\alpha \in \mathbb{R}. 
\end{align}
\end{enumerate}
We remark that the above assumptions imply that there exists $\Lambda\in \R$ such that $y\mapsto \E_{\varepsilon}(y)-\Lambda \mathcal{R}_{\varepsilon}(y)$ is convex, i.e. $\E\e$ is $\Lambda$-convex w.r.t. $\rcal\e$. In particular, if $\lambda<0$, then we set $\Lambda=\lambda c$, and in the case $\lambda\geq 0$, $\Lambda= \frac{\lambda}{c}$. 

Let $T>0$ be a finite time horizon. We consider the evolutionary variational inequality (EVI) formulation of the gradient flow $\brac{Y,\E\e,\rcal\e}$: Find $y\in H^1(0,T;Y)$ such that for a.e. $t\in (0,T)$,
\begin{equation}\label{eq:317}
\expect{D\rcal\e(\dot{y}(t)),y(t)-\widetilde{y}}_{Y^*,Y}\leq \E\e(\widetilde{y})-\E\e(y(t))-\Lambda \rcal\e(y(t)-\widetilde{y}) \quad \text{for all }\widetilde{y}\in Y.
\end{equation}

\begin{remark}[Existence and uniqueness]
Assumptions \ref{C1:ass:a}-\ref{C3:ass:a} imply that $\E\e$ is proper, l.s.c., coercive and $\Lambda$-convex w.r.t. $\rcal\e$. In this respect, the classical theory of maximal monotone operators with Lipschitz perturbations implies that for an initial datum $y^0\in \mathrm{dom}(\E\e)$, there exists a unique $y\in H^1(0,T;Y)$ which satisfies \eqref{eq:317} and $y(0)=y^0$, see \cite{brezis1973ope,barbu2010nonlinear}, where the \textit{Yosida regularization} technique is used for the proof of this result. In view of the continuous embedding $H^1(0,T;Y)\subset C([0,T],Y)$, we identify functions in $H^1(0,T;Y)$ by their continuous representatives. Moreover, the following standard apriori estimate holds 
\begin{equation}\label{eq:323}
\int_{0}^t \rcal\e(\dot{y}(s))ds \leq \E\e(y^0)-\E\e(y(t)) \quad  \text{for all }t \in [0,T],
\end{equation} 
which follows by testing \eqref{eq:145} with $\dot{y}(s)$ and by the chain rule for the $\Lambda$-convex functional $\E\e$. \eqref{eq:323} in combination with the growth conditions \eqref{eq:271} and \eqref{eq:276} yields
\begin{equation}\label{eq:295}
\norm{y(t)}^p_{L^p(\Omega)\otimes W^{1,p}_0(Q)}+\norm{y(t)}^{\theta}_{L^{\theta}(\Omega\times Q)} \leq c \brac{\E\e(y^0)+2c}. 
\end{equation} 
\end{remark}
\medskip

\noindent \textbf{Effective system.} In the limit $\varepsilon\to 0$, we derive an effective gradient flow which is described as follows. The state space is given by
$$Y_0=L^2_{\mathrm{inv}}(\Omega)\otimes L^2(Q).$$
The effective dissipation potential is given by $\mathcal{R}_{\mathrm{hom}}: Y_0\to [0,\infty)$,
\begin{equation*}
\mathcal{R}_{\mathrm{hom}}(\dot{y})=\expect{\int_{Q}r(\omega,x) |\dot{y}(\omega,x)|^2 dx}.
\end{equation*}
The energy functional is $\E_{\mathrm{hom}}: Y_{0} \to \R\cup \cb{\infty}$, 
\begin{align}\label{equation_442}
\begin{split}
\E_{\mathrm{hom}}(y)  =  & \inf_{\chi \in L^p_{\mathrm{pot}}(\Omega)\otimes L^p(Q)}\expect{\int_{Q} V\brac{\omega,x,\nabla y(\omega,x)+\chi(\omega,x)}dx}\\ & +\expect{\int_{Q}   f(\omega,x,y(\omega,x))dx} 
\end{split}
\end{align}
for $y \in (L^p_{\mathrm{inv}}(\Omega)\otimes W^{1,p}_0(Q))\cap \brac{L^{\theta}_{\mathrm{inv}}(\Omega)\otimes L^{\theta}(Q)}=: \mathrm{dom}(\E_{\mathrm{hom}})$ and $\E_{\mathrm{hom}}=\infty$ otherwise. We remark that $\E\h(\cdot)-\Lambda \rcal\h(\cdot)$ is convex with the same $\Lambda\in \R$ as for $\E\e$. 

The gradient flow $\brac{Y_0,\E\h,\rcal\h}$ in the EVI formulation also admits a unique solution, i.e., for an initial datum $y^0\in \mathrm{dom}(\E\h)$, there exists a unique $y\in H^1(0,T; Y_0)$ such that $y(0)=y^0$ and for a.e. $t\in (0,T)$,
\begin{equation}\label{eq:318}
\expect{D\rcal\h(\dot{y}(t)),y(t)-\widetilde{y}}_{Y^*_0,Y_0}\leq \E\h(\widetilde{y})-\E\h(y(t))-\Lambda \rcal\h(y(t)-\widetilde{y}), 
\end{equation}
$\text{for all }\widetilde{y}\in Y_0$.

The main result of this paper is the following homogenization theorem. In particular, the proof relies on the modified abstract strategy discussed in the introduction and on the stochastic unfolding procedure that is explained in Section \ref{S_Stoch}.
\begin{theorem}[Homogenization]\label{s3_thm_5} Let $p\in (1,\infty)$, $\theta \in [2,\infty)$ and $Q\subset \R^d$ be open and bounded. Assume \ref{C1:ass:a}-\ref{C3:ass:a}, and consider $y^0\in \mathrm{dom}(\E_{\mathrm{hom}})$, $y^0_{\varepsilon} \in \mathrm{dom}(\E\e)$ such that, as $\varepsilon\to 0$,
\begin{equation*}
y_{\varepsilon}^0 \to y^0 \quad\text{strongly in }Y, \quad \limsup_{\varepsilon\to 0}\E_{\varepsilon}(y^0_{\varepsilon}) < \infty.  
\end{equation*}
Let $y\e\in H^1(0,T; Y)$ be the unique solution to the EVI \eqref{eq:317} with $y\e(0)=y\e^0$. Then, for all $t\in (0,T]$, as $\varepsilon \to 0$, 
\begin{align*}
& y\e(t) \to y(t) \quad \text{strongly in }Y,
\end{align*}
where $y\in H^1(0,T; Y_0)$ is the unique solution to the EVI \eqref{eq:318} with $y(0)=y^0$. Moreover, if we additionally assume that $\E_{\varepsilon}(y^{0}\e)\to \E_{\mathrm{hom}}(y^0)$, then it holds that $\dot{y}\e \to \dot{y}$ strongly in $L^2(0,T; Y)$ and $\E_{\varepsilon}(y\e(t))\to \E_{\mathrm{hom}}(y(t))$ for all $t\in [0,T]$.

\noindent
\textit{(For the proof see Section \ref{Section_3.2}.)}
\end{theorem}

\begin{remark}[Convergence of gradients]
We remark that in the proof we additionally show that $y\e(t)\overset{2}{\harpoon}y(t)$ in $L^{\theta}(\Omega\times Q)$ and in $L^p(\Omega\times Q)$, where ``$\overset{2}{\harpoon}$'' is weak stochastic two-scale convergence in the mean defined in Definition \ref{def46}. Also, it holds $P_{\mathrm{inv}}\nabla y\e(t)\harpoon \nabla y(t)$ weakly in $L^p(\Omega\times Q)^d$. If we additionally assume that $V(\omega,x,\cdot)$ is strictly convex, we may obtain that for all $t\in (0,T]$ it holds
\begin{equation*}
\nabla y\e(t)\overset{2}{\harpoon} \nabla y(t)+\chi(t) \quad \text{in }L^p(\Omega\times Q)^d,
\end{equation*}
where $\chi(t)\in L^p_{\mathrm{pot}}(\Omega)\otimes L^p(Q)$ is the unique minimizer in the corrector problem 
\begin{equation*}
\inf_{\chi\in L^p\p(\Omega)\otimes L^p(Q)}\expect{\int_{Q}V(\omega,x,\nabla y(t,\omega,x)+\chi(\omega,x))dx}.
\end{equation*}
\end{remark}
\begin{remark}[Ergodic case]\label{rem:82:c}
If we additionally assume that $\expect{\cdot}$ is ergodic, the limit system is driven by deterministic functionals. In particular, the state space reduces to $Y_0=L^2(Q)$. The dissipation potential is given by
\begin{equation*}
{\mathcal{R}}_{\mathrm{hom}}(\dot{y})=\int_Q r\h(x)|\dot{y}(x)|^2 dx, 
\end{equation*}
where $r\h(x)=\expect{r(\omega,x)}$. The energy functional boils down to 
\begin{equation*}
\E_{\mathrm{hom}}(y)=\int_{Q} V_{\mathrm{hom}}\brac{x,\nabla y(x)}+ f_{\mathrm{hom}}(x,y(x)) dx
\end{equation*}
in $W^{1,p}_0(Q)\cap L^{\theta}(Q)$ and otherwise $\infty$. Above, $f_{\mathrm{hom}}(x,\alpha)=\expect{f(\omega,x,\alpha)}$ for $x\in Q$ and $\alpha \in \R$, and  $V_{\mathrm{hom}}(x,F)=\inf_{\chi \in L^p_{\mathrm{pot}}(\Omega)}\expect{V(x,\omega,F+\chi(\omega))}$ for $x\in Q$, $F\in \R^d$. Moreover, $V\h$ satisfies analogous $p$-growth conditions as $V$. The identification of $\E\h$ can be obtained by a measurable selection argument from Remark \ref{rem:901:a} (cf. proof of Lemma \ref{lem:128:a}).
\end{remark}

\section{Stochastic unfolding method}\label{S_Stoch}
In this section we introduce the stochastic unfolding method. In particular, in Section \ref{S_Stoch_1} we define the unfolding operator and present its main properties. In Section \ref{S_Two} we obtain weak two-scale type compactness statements and we construct suitable recovery sequences. To keep the exposition simple, the proofs are presented in the end, in Section \ref{S_Proofs}.
\subsection{Stochastic unfolding operator and two-scale convergence in the mean}\label{S_Stoch_1}

\begin{lemma}\label{L:unf}
Let $\varepsilon>0$, $p\in (1,\infty)$, $q=\frac{p}{p-1}$, and $Q\subset\R^d$ be open. There exists a unique linear isometric isomorphism
  \begin{equation*}
    \unf: L^p(\Omega \times Q) \rightarrow L^p(\Omega \times Q)
  \end{equation*}
  which satisfies
  \begin{equation*}
    \text{for all } u\in L^p(\Omega)\overset{a}{\otimes} L^p(Q),  \qquad (\unf u)(\omega,x)=u(\tau_{-\frac{x}{\varepsilon}}\omega,x)\qquad \text{a.e. in }\Omega\times Q.
  \end{equation*}
  Moreover, its adjoint is the unique linear isometric isomorphism $\unf^{*}:L^q(\Omega \times Q) \to L^q(\Omega \times Q)$ that satisfies for all $u\in L^q(\Omega)\overset{a}{\otimes}L^q(Q)$, $(\unf^{*}u)(\omega,x)=u(\tau_{\frac{x}{\varepsilon}}\omega,x)$ a.e.~in $\Omega\times Q$.

\noindent(For the proof see Section \ref{S_Proofs}.)
\end{lemma}
\begin{definition}[Unfolding operator and two-scale convergence in the mean]\label{def46}
The operator $\unf:L^p(\Omega \times Q)\to L^p(\Omega \times Q)$ from Lemma~\ref{L:unf} is called the stochastic unfolding operator. We say that a sequence $(u\e) \subset L^p(\Omega \times Q)$ weakly (strongly) two-scale converges in the mean in $L^p(\Omega \times Q)$ to $u\in L^p(\Omega \times Q)$ if, as $\varepsilon\to 0$,
  \begin{equation*}
    \unf u\e \rightarrow u \quad \text{ weakly (strongly) in }L^p(\Omega \times Q).
  \end{equation*} 
  In this case we write $u\e\wt u$ (resp. $u_{\varepsilon} \st u$) in $L^p(\Omega\times Q)$.
\end{definition}

The below lemma directly follows from the isometry property of $\unf$ and the usual properties of weak and strong convergence in $L^p(\Omega\times Q)$; therefore, we do not present its proof.

\begin{lemma}[Basic properties]\label{lemma_basics} Let $p\in (1,\infty)$, $q=\frac{p}{p-1}$ and $Q\subset \R^d$ be open.
Consider sequences $(u\e)$ in $L^p(\Omega \times Q)$ and $(v\e)$ in $L^q(\Omega \times Q)$.
\begin{enumerate}[label=(\roman*)]
\item If $u\e \overset{2}{\rightharpoonup} u$ in $L^p(\Omega\times Q)$, then $\sup_{\varepsilon\in (0,1)}\norm{u\e}_{L^p(\Omega\times Q)}< \infty$ and 
\begin{equation*}
\norm{u}_{L^p(\Omega\times Q)}\leq \liminf_{\varepsilon\to 0}\norm{u\e}_{L^p(\Omega\times Q)}.
\end{equation*}
\item If $\limsup_{\varepsilon\rightarrow 0}\norm{u\e}_{L^p(\Omega\times Q)}<\infty$, then there exist a subsequence $\varepsilon'$ and $u\in L^p(\Omega \times Q)$ such that $u_{\varepsilon'} \overset{2}{\harpoon} u$ in $L^p(\Omega \times Q)$.
\item  $u\e \overset{2}{\to} u$ in $L^p(\Omega\times Q)$ if and only if $u\e \overset{2}{\harpoon} u$ in $L^p(\Omega\times Q)$ and $\norm{u\e}_{L^p(\Omega\times Q)} \to \norm{u}_{L^p(\Omega\times Q)}$.
\item  If $u\e \overset{2}{\harpoon} u$ in $L^p(\Omega \times Q)$ and $v\e \overset{2}{\to} v$ in $L^q(\Omega \times Q)$, then
\begin{equation*}
\expect{\int_Q u\e(\omega,x) v\e (\omega,x) dx}\rightarrow \expect{\int_Q u(\omega,x)v(\omega,x) dx}.
\end{equation*}
\end{enumerate}
\end{lemma}

For homogenization of variational problems, in particular problems driven by convex integral functionals, the following transformation and (lower semi-)continuity properties are very useful.

\begin{proposition}\label{P_Cont_1}
Let $p\in (1,\infty)$ and $Q \subset \R^d$ be open and bounded. Let $V: \Omega \times Q \times \R^{m}\to \R$ be such that $V(\cdot,\cdot,F)$ is $\mathcal{F} \otimes \mathcal{L}(Q)$-measurable for all $F\in \R^m$ and $V(\omega,x,\cdot)$ is continuous for a.e. $(\omega,x)\in \Omega \times Q$. Also, we assume that there exists $c>0$ such that for a.e. $(\omega,x)\in \Omega\times Q$
\begin{equation*}
|V(\omega, x, F)|\leq c(1+|F|^p), \quad \text{for all }F\in \R^m.
\end{equation*}
\begin{enumerate}[label=(\roman*)]
\item For all $u\in L^p(\Omega \times Q)^{m}$, we have
\begin{equation}\label{intform}
\expect{\int_Q V(\tau_{\frac{x}{\varepsilon}}\omega, x,u(\omega,x))dx}=\expect{\int_Q V(\omega, x, \unf u(\omega,x))dx} \,.
\end{equation}
\item 
If $u\e \overset{2}{\to} u$ in $L^p(\Omega\times Q)^m$, then 
\begin{equation*}
\lim_{\varepsilon\to 0}\expect{\int_{Q}V(\tau_{\frac{x}{\varepsilon}}\omega,x, u\e(\omega,x))dx} = \expect{\int_{Q} V(\omega, x, u(\omega,x))dx}.
\end{equation*} 
\item We additionally assume that for a.e. $(\omega,x)\in \Omega\times Q$, $V(\omega, x,\cdot)$ is convex. Then, if $u\e \overset{2}{\harpoon} u$ in $L^p(\Omega\times Q)^m$,
\begin{equation*}
\liminf_{\varepsilon\to 0}\expect{\int_{Q} V(\tau_{\frac{x}{\varepsilon}} \omega,x,u\e(\omega,x)) dx}\geq \expect{\int_{Q} V(\omega, x, u(\omega,x))dx}.
\end{equation*}
\end{enumerate}
\noindent(For the proof see Section \ref{S_Proofs}.)
\end{proposition}

\begin{remark}[{Comparison to the notion of} \cite{bourgeat1994stochastic}]\label{R_Stoch_1}
The notion of weak two-scale convergence in the mean of Definition~\ref{def46}, i.e., weak convergence of unfolded sequences, coincides with the convergence notion introduced in \cite{bourgeat1994stochastic} (see also \cite{andrews1998stochastic}). More precisely, for a bounded sequence $(u_{\varepsilon})\subset L^p(\Omega\times Q)$ we have $u\e\wt u$ in $L^p(\Omega\times Q)$ (in the sense of Definition~\ref{def46}) if and only if $u\e$ stochastically two-scale converges in the mean to $u$ in the sense of \cite{bourgeat1994stochastic}, i.e. 
\begin{equation}\label{eq:1}
\lim_{\varepsilon\rightarrow 0}\expect{\int_{Q}u_{\varepsilon}(\omega,x)\varphi(\tau_{\frac{x}{\varepsilon}}\omega,x)dx}=\expect{ \int_{Q}u(\omega,x)\varphi(\omega,x)dx},
\end{equation}
for any $\varphi\in L^q(\Omega\times Q)$ that is admissible (in the sense that the mapping $(\omega,x)\mapsto \varphi(\tau_{\frac{x}{\varepsilon}}\omega,x)$ is well-defined). Indeed, with help of $\unf$ (and its adjoint) we might rephrase the integral on the left-hand side in \eqref{eq:1} as
\begin{equation}\label{eq:1234}
\expect{\int_{Q}u_{\varepsilon}(\unf^{*}\varphi)\, dx}=\expect{\int_{Q}(\unf u_{\varepsilon})\varphi dx},
\end{equation}
which proves the equivalence. For the reason of this equivalence, we use the terms \textit{weak} and \textit{strong stochastic two-scale convergence in the mean} instead of talking about \textit{weak} or \textit{strong convergence of unfolded sequences}. 

The arguments in this paper are inspired by both, the \textit{unfolding approach}---we transform intregrals with oscillations into integrals without (or controlable) oscillations---and \textit{two-scale convergence} in the sense that we make use of oscillating test-functions.
\end{remark}
\subsection{Two-scale limits of gradients}\label{S_Two}
\def\per{{\sf per}}
 
The following proposition presents a weak two-scale compactness statement for sequences of gradient fields.

\begin{proposition}[Compactness]\label{prop1}
Let $p\in (1,\infty)$ and $Q\subset \R^d$ be open. Let $(u\e)$ be a bounded sequence in $L^p(\Omega)\otimes W^{1,p}(Q)$. Then, there exist $u\in L^p_{{\mathrm{inv}}}(\Omega)\otimes W^{1,p}(Q)$ and $\chi\in L^p_{\mathrm{pot}}(\Omega)\otimes L^p(Q)$ such that, up to a subsequence,
\begin{equation}\label{equation1}
u\e \overset{2}{\harpoon} u  \quad \text{in }L^p(\Omega\times Q), \quad \nabla u\e \overset{2}{\harpoon} \nabla u +\chi  \quad \text{in }L^p(\Omega\times Q)^d.
\end{equation}
If, additionally, $\expect{\cdot}$ is ergodic, then $u=P_{\mathrm{inv}} u=\expect{u} \in W^{1,p}(Q)$ and $\expect{u_{\varepsilon}}\harpoon u$ weakly in $W^{1,p}(Q)$.

\noindent(For the proof see Section \ref{S_Proofs}.)
\end{proposition}
We remark that the above result is already established in \cite{bourgeat1994stochastic} in the context of two-scale convergence in the mean in the $L^2$-space setting. We recapitulate its short proof from the perspective of stochastic unfolding, see Section \ref{S_Proofs}.

\begin{remark}\label{R_Two_1} Note that the proof of the above proposition reveals that $P_{\mathrm{inv}}u\e \harpoon u$ weakly in $L^p_{\mathrm{inv}}(\Omega)\otimes W^{1,p}(Q)$ (see Lemma \ref{lem7}). If we consider a closed subspace $X\subset W^{1,p}(Q)$ and assume that $u\e(\omega)\in X$ $P$-a.e., then $P_{\mathrm{inv}}u\e \in L^p_{\mathrm{inv}}(\Omega)\otimes X$. Therefore, it follows that $u\in L^p_{\mathrm{inv}}(\Omega)\otimes X$. This observation is useful if we consider boundary value problems, e.g., if $X=W^{1,p}_0(Q)$. We may argue similarly for closed convex subsets in $W^{1,p}(Q)$.
\end{remark}

\begin{lemma}[Recovery sequence]\label{Nonlinear_recovery} Let $p,\theta \in (1,\infty)$ and $Q\subset \R^d$ be open. For $\chi\in L^p_{\mathrm{pot}}(\Omega)\otimes L^p(Q)$ and $\delta>0$, there exists a sequence $g_{\delta,\varepsilon}(\chi) \in L^p(\Omega)\otimes W^{1,p}_0(Q)$ such that
\begin{equation*}
\|g_{\delta,\varepsilon}(\chi)\|_{L^{\theta}(\Omega\times Q)} \leq \varepsilon c(\delta), \quad \limsup_{\varepsilon\to 0}\|\mathcal{T}_{\varepsilon}\nabla g_{\delta, \varepsilon}(\chi)-\chi\|_{L^p(\Omega\times Q)^d}\leq \delta,
\end{equation*}
where $c(\delta)>0$ does not depend on $\varepsilon$.

\noindent
(For the proof see Section \ref{S_Proofs}.)
\end{lemma} 
\subsection{Proofs of the statements in Section \ref{S_Stoch}}\label{S_Proofs}
Before presenting the proofs, we recall some basic facts from functional analysis which will be helpful in the following.

\begin{remark}
Let $p \in (1,\infty)$ and $q= \frac{p}{p-1}$.
  \begin{enumerate}[label=(\roman*)]
  \item $\expect{\cdot}$ is ergodic $\Leftrightarrow$ $L^p_{\mathrm{inv}}(\Omega)\simeq \R$ $\Leftrightarrow$ $P_{\mathrm{inv}} f=\expect{f}$. 
  \item The following orthogonality relations hold (for a proof see \cite[Section 2.6]{brezis2010functional}): We identify the dual space $L^p(\Omega)^*$ with $L^q(\Omega)$, and define for a set $A\subset L^q(\Omega)$ its orthogonal complement $A^{\bot}\subset L^p(\Omega)$ as 
\begin{equation*} 
A^{\bot}=\cb{\varphi\in L^p(\Omega):\expect{\varphi \psi}=0 \text{ for all }\psi\in A}.
\end{equation*}  
It holds
    \begin{equation}\label{orth}
      \mathrm{ker}(D)=\mathrm{ran}(D^*)^{\bot}, \quad 
      L^p_{\mathrm{pot}}(\Omega)= \overline{\mathrm{ran}(D)}=\mathrm{ker}(D^*)^{\bot}.
    \end{equation}
    Above, $\mathrm{ker}(\cdot)$ denotes the kernel and $\mathrm{ran}(\cdot)$ the range of an operator.
  \end{enumerate}
\end{remark}
\begin{proof}[Proof of Lemma \ref{L:unf}]
We first define $\unf$ on $\mathcal{A}:=\{u(\omega,x)=\varphi(\omega)\eta(x)\,:\,\varphi\in L^p(\Omega),\,\eta\in L^p(Q)\,\}\subset L^p(\Omega\times Q)$ by setting $(\unf u)(\omega,x)=\varphi(\tau_{-\frac{x}{\varepsilon}}\omega) \eta(x)$ for all $u=\varphi\eta\in\mathcal{A}$. In view of Assumption \ref{Assumption_2_1} (iii), $\unf u$ is $\mathcal F\otimes\mathcal L(Q)$-measurable and using the measure preserving property of $\tau$, we have
  \begin{equation*}
    \|\unf u\|_{L^p(\Omega\times Q)}^p=\int_Q\expect{|\varphi(\tau_{-\frac{x}{\varepsilon}}\omega)|^p}|\eta(x)|^p\,dx=\|\varphi\|_{L^p(\Omega)}^p\|\eta\|_{L^p(Q)}^p=\|u\|_{L^p(\Omega\times Q)}^p.
  \end{equation*}
Since $\mbox{span}(\mathcal{A})$ is dense in $L^p(\Omega\times Q)$, $\unf$ extends to a linear isometry from $L^p(\Omega\times Q)$ to $L^p(\Omega\times Q)$. We define a linear isometry $\mathcal{T}_{-\varepsilon}: L^q(\Omega\times Q)\to L^q(\Omega\times Q)$ analogously as $\mathcal{T}_{\varepsilon}$, with $\varepsilon$ replaced by $-\varepsilon$. Then for any $\varphi\in L^p(\Omega)\stackrel{a}{\otimes} L^p(Q)$ and $\psi\in L^q(\Omega)\stackrel{a}{\otimes} L^q(Q)$ we have (thanks to the measure preserving property of $\tau$  and Fubini):
  \begin{eqnarray*}
    \expect{\int_Q(\unf\varphi)\psi\,dx}&=&\int_Q\expect{\varphi(\tau_{-\frac{x}{\varepsilon}}\omega,x)\psi(\omega,x)}dx\\
  &=&\int_Q\expect{\varphi(\omega,x)\psi(\tau_{\frac{x}{\varepsilon}}\omega,x)}dx=\expect{\int_Q\varphi(\mathcal T_{-\varepsilon}\psi)dx}.
  \end{eqnarray*}
  Since  $L^p(\Omega)\stackrel{a}{\otimes} L^p(Q)$ and $L^q(\Omega)\stackrel{a}{\otimes} L^q(Q)$ are dense in $L^p(\Omega\times Q)$ and $L^q(\Omega\times Q)$, respectively, we conclude that $\unf^*=\mathcal T_{-\varepsilon}$. Since $\mathcal{T}_{\varepsilon}^*$ is an isometry, it follows that $\unf$ is surjective (see \cite[Theorem 2.20]{brezis2010functional}). Analogously, $\mathcal{T}_{\varepsilon}^*$ is also surjective.
\end{proof}
\begin{proof}[Proof of Proposition \ref{P_Cont_1}]
We first note that $V$ is a Carath{\'e}odory integrand in the sense of Remark \ref{rem:appen} (if necessary we tacitly redefine it by $V(\omega,x,\cdot)=0$ for $(\omega,x)$ in a set of measure $0$) and therefore it follows that $V$ is a normal integrand (see Appendix \ref{sec:norm:a}). For fixed $\varepsilon>0$, the mapping $(\omega,x)\mapsto (\tau_{\frac{x}{\varepsilon}}\omega,x)$ is $\brac{\mathcal{F}\otimes \mathcal{L}(Q),\mathcal{F}\otimes \mathcal{L}(Q)}$-measurable and therefore $(\omega,x,F)\mapsto V(\tau_{\frac{x}{\varepsilon}}\omega,x,F)$ defines as well a Carath{\'e}odory and thus normal integrand. Hence, with the help of the growth condition, all the integrals in the statement of the proposition are well-defined.

\textit{Proof of (i):} We first consider the case $u\in L^p(\Omega)\overset{a}{\otimes} L^p(Q)^m$. By Fubini's theorem, the measure preserving property of $\tau$, and  by the transformation $\omega\mapsto \tau_{-\frac{x}{\varepsilon}}\omega$, we have
\begin{align*}
\expect{\int_{Q}V(\tau_{\frac{x}{\varepsilon}}\omega,x,u(\omega,x))dx} & =\int_{Q}\expect{V(\tau_{\frac{x}{\varepsilon}}\omega,x,u(\omega,x))}dx \\ & = \int_{Q}\expect{V(\omega,x,u(\tau_{-\frac{x}{\varepsilon}}\omega,x))}dx.
\end{align*}
Since $u\in L^p(\Omega)\stackrel{a}{\otimes}L^p(Q)$, we have $u(\tau_{-\frac{x}{\varepsilon}}\omega,x)=\unf u(\omega,x)$, and thus \eqref{intform} follows. The general case follows by an approximation argument. Indeed, for any $u\in L^p(\Omega\times Q)^m$ we can find a sequence $u_k \in L^p(\Omega)\overset{a}{\otimes} L^p(Q)^m$ such that $u_k \to u$ strongly in $L^p(\Omega\times Q)^m$, and by passing to a subsequence (not relabeled) we may additionally assume that $u_k\to u$ pointwise a.e.~in $\Omega\times Q$. 
By continuity of $V$ in its last variable, we thus have $V(\tau_{\frac{x}{\varepsilon}}\omega,x, u_{k}(\omega,x)) \to V(\tau_{\frac{x}{\varepsilon}}\omega,x, u(\omega,x))$ for a.e.~$(\omega,x)\in \Omega\times Q$. Since $|V(\tau_{\frac{x}{\varepsilon}}\omega,x, u_{k}(\omega,x))|\leq c(1+|u_{k}(\omega,x)|^p)$ a.e.~in $\Omega\times Q$, the dominated convergence theorem (\cite[Theorem 2.8.8]{bogachev2007measure}) implies that $$\lim_{k\to \infty}\expect{\int_{Q}V(\tau_{\frac{x}{\varepsilon}}\omega,x,u_{k}(\omega,x))dx} = \expect{\int_{Q}V(\tau_{\frac{x}{\varepsilon}}\omega,x,u(\omega,x))dx}.$$ In the same way we conclude that $$\lim_{k\to \infty}\expect{\int_{Q}V(\omega,x,\unf u_k(\omega,x))dx} = \expect{\int_{Q}V(\omega,x,\unf u(\omega,x))dx}.$$ Since the integrals on the left-hand sides are the same, \eqref{intform} follows.

\textit{Proof of (ii):} We obtain $\expect{\int_Q V(\tau_{\frac{x}{\varepsilon}}\omega, x,u_{\varepsilon}(\omega,x))dx}=\expect{\int_Q V(\omega, x, \unf u_{\varepsilon}(\omega,x))dx}$ using part (i). Since by assumption $\unf u\e \to u$ strongly in $L^p(\Omega\times Q)^m$, using the growth conditions of $V$ and the dominated convergence theorem, it follows, similarly as in part (i), that we have $\lim_{\varepsilon\to 0}\expect{\int_Q V(\omega, x, \unf u_{\varepsilon}(\omega,x))dx}= \expect{\int_{Q}V(\omega,x,u(\omega,x))dx}$.

\textit{Proof of (iii):} The functional $L^p(\Omega\times Q)^m\ni u \mapsto \expect{\int_{Q}V(\omega, x,u(\omega,x))dx}$ is convex and lower semi-continuous, therefore it is weakly lower semi-continuous (see \cite[Corollary 3.9]{brezis2010functional}). Combining this fact with the transformation formula from (i) and the weak convergence $\unf u_{\varepsilon}\harpoon u$ (by assumption), the claim follows.
\end{proof}
Before stating the proof of Proposition \ref{prop1}, we present some auxiliary lemmas.
\begin{lemma}\label{lemA} Let $p \in (1,\infty)$ and $q=\frac{p}{p-1}$.
\begin{enumerate}[label=(\roman*)]
\item
If $\varphi\in \cb{D^*\psi:\psi\in W^{1,q}(\Omega)^d}^{\bot}$, then $\varphi\in L^p_{{\mathrm{inv}}}(\Omega)$.
\item
If $\varphi \in \cb{\psi\in W^{1,q}(\Omega)^d: D^*\psi=0}^{\bot}$, then $\varphi\in L^p_{\mathrm{pot}}(\Omega)$.
\end{enumerate}
\end{lemma}
\begin{proof}
\textit{Proof of (i).} First, we note that 
\begin{align*}
\varphi \in L^p_{\mathrm{inv}}(\Omega) \quad \Leftrightarrow \quad U_{h e_i}U_{y}\varphi=U_y\varphi  \quad \text{for all }y\in \R^d,h\in \R, i=1,...,d.
\end{align*} 
We consider $\varphi\in \cb{D^*\psi:\psi\in W^{1,q}(\Omega)^d}^{\bot}$ and we show that $\varphi\in L^p_{\mathrm{inv}}(\Omega)$ using the above equivalence. 
Let $\psi \in W^{1,q}(\Omega)$ and $i\in \cb{1,...,d}$. Then by the group property we have $U_{-h e_i}\psi-\psi=\int_{0}^h U_{-t e_i}D_i^*\psi dt$ and therefore
\begin{align*}
\expect{(U_{h e_i}\varphi-\varphi)\psi}=\expect{\varphi (U_{-he_i}\psi-\psi)}=\langle \varphi \int_{0}^h U_{-t e_i}D_i^*\psi dt\rangle=\int_{0}^h\expect{\varphi D^*_i(U_{-t e_i}\psi)}dt.
\end{align*}
Since $U_{-t e_i}\psi \in W^{1,q}(\Omega)$ for any $t\in [0,h]$, we obtain $\expect{\varphi D^*_i(U_{-t e_i}\psi)}=0$ and thus $U_{h e_i}\varphi = \varphi$. Furthermore, for any $y\in \R^d$, we have $\expect{(U_{h e_i}U_y \varphi - U_y \varphi)\psi}=\expect{(U_{h e_i}\varphi -\varphi)U_{-y}\psi}=0$ by the same argument.

\textit{Proof of (ii).} In view of $L^p_{\mathrm{pot}}(\Omega)=\mathrm{ker}(D^*)^{\bot}$ (see (\ref{orth})), it is sufficient to prove that the set $\cb{\varphi\in W^{1,q}(\Omega)^d: D^*\varphi=0}$ is dense in $\mathrm{ker} (D^*)$. This follows by an approximation argument as in \cite[Section 7.2]{jikov2012homogenization}.
Let $\varphi \in \mathrm{ker}(D^*)$ and we define for $t>0$
\begin{equation*}
\varphi^t(\omega)=\int_{\R^{d}}p_t(y)\varphi(\tau_y \omega)dy, \quad \text{where }p_t(y)=\frac{1}{\brac{4\pi t}^{\frac{d}{2}}}e^{-\frac{|y|^2}{4t}}.
\end{equation*}
Then the claimed density follows, since $\varphi^t\in W^{1,q}(\Omega)^d$, $D^*\varphi^t=0$ for any $t>0$ and $\varphi^t \rightarrow \varphi$ strongly in $L^q(\Omega)^d$ as $t\to 0$. The last statement can be seen as follows. By the continuity property of $U_y$, for any $\varepsilon>0$ there exists $\delta>0$ such that $\expect{|\varphi(\tau_y \omega)-\varphi(\omega)|^q}\leq \varepsilon$ for any $y\in B_{\delta}(0)$.
It follows that
\begin{align*}
\expect{|\varphi^t-\varphi|^q} & =\expect{\bigg|\int_{\R^{d}}p_t(y)\brac{\varphi(\tau_y \omega)-\varphi(\omega)}dy\bigg|^q}\\  & \leq \int_{\R^{d}}p_t(y)\expect{|\varphi(\tau_y \omega)-\varphi(\omega)|^q}dy \\ & =   \int_{B_{\delta}}p_t(y)\expect{|\varphi(\tau_y \omega)-\varphi(\omega)|^q}dy+\int_{\R^{d}\setminus B_{\delta}}p_t(y)\expect{|\varphi(\tau_y \omega)-\varphi(\omega)|^q}dy.
\end{align*}
The first term on the right-hand side of the above inequality is bounded by $\varepsilon$ as well as the second term for sufficiently small $t>0$.
\end{proof}
\begin{lemma}\label{lem6}
Let $p\in (1,\infty)$ and $Q\subset \R^d$ be open. Let $u\e \in L^p(\Omega)\otimes W^{1,p}(Q)$ be such that $u\e \overset{2}{\harpoon} u$ in $L^p(\Omega \times Q)$ and $\varepsilon \nabla u\e \overset{2}{\harpoon} 0$ in $L^p(\Omega \times Q)^d$. Then $u\in L^p_{{\mathrm{inv}}}(\Omega)\otimes L^p(Q)$.
\end{lemma}
\begin{proof}
Consider a sequence $v\e=\varepsilon \mathcal{T}_{\varepsilon}^*(\varphi \eta)$ such that $\varphi\in W^{1,q}(\Omega)$ and $\eta\in C^{\infty}_c(Q)$. Note that $\mathcal{T}\e v\e=\varepsilon \varphi \eta$ and we have, for $i=1,...,d$ and as $\varepsilon\to 0$,
\begin{equation*}
\expect{\int_Q \partial_i u\e v\e dx}=\expect{\int_Q(\mathcal{T}\e \partial_i u\e) (\mathcal{T}\e v\e) dx}=\expect{\int_Q(\mathcal{T}\e \partial_i u\e) \varepsilon \varphi \eta dx}\rightarrow 0.
\end{equation*}
Moreover, it holds that $\partial_{i}v_{\varepsilon}= \mathcal{T}_{\varepsilon}^*(D_i \varphi \eta + \varepsilon \varphi \partial_i \eta)$ and therefore
\begin{align*}
\expect{\int_Q\partial_i u\e v\e dx}  =-\expect{\int_Q u\e \partial_i v\e dx} & =-\expect{\int_Q u\e \mathcal{T}_{\varepsilon}^*(D_i \varphi \eta + \varepsilon \varphi \partial_i \eta)dx}\\ &=
-\expect{\int_Q (\mathcal{T}\e u\e) D_i\varphi \eta+\varepsilon (\mathcal{T}\e u\e) \varphi \partial_i\eta dx}. 
\end{align*}
The last expression converges to $-\expect{\int_Q u D_i\varphi \eta dx}$ as $\varepsilon\to 0$.
As a result of this, $\expect{u(x)D_i\varphi}=0$ for almost every $x\in Q$ and therefore $u\in L^p_{{\mathrm{inv}}}(\Omega)\otimes L^p(Q)$ by Lemma \ref{lemA} (i).
\end{proof}
\begin{lemma}\label{lem7}
Let $p\in (1,\infty)$ and $Q\subset \R^d$ be open. Let $u\e$ be a bounded sequence in $L^p(\Omega)\otimes W^{1,p}(Q)$. Then there exists $u\in L^p_{{\mathrm{inv}}}(\Omega)\otimes W^{1,p}(Q)$ such that (up to a subsequence)
\begin{equation*}
u\e \overset{2}{\harpoon} u \text{ in }L^p(\Omega \times Q), \quad \pinv u\e\overset{2}{\harpoon} u \text{ in }L^p(\Omega \times Q),\quad \pinv \nabla u\e \overset{2}{\harpoon} \nabla u \text{ in }L^p(\Omega \times Q)^d.
\end{equation*}
In particular, it holds that $\pinv u\e \harpoon u$ weakly in $L^p_{\mathrm{inv}}(\Omega)\otimes W^{1,p}(Q)$.
\end{lemma}
\begin{proof}
\textit{Step 1. Proof of the identity $\pinv \circ \mathcal{T}\e=\mathcal{T}\e\circ \pinv =\pinv$.} The second identity holds by definition of $\pinv$. To show that  $\pinv \circ \mathcal{T}\e=\pinv$, we consider $v\in L^p(\Omega \times Q)$, $\varphi\in L^q(\Omega)$ and $\eta\in L^q(Q)$. We have
\begin{align*}
\expect{\int_Q (\pinv\unf v) (\varphi \eta) dx} =\expect{\int_Q (\unf v) \pinv^* (\varphi \eta) dx} & =\expect{\int_Q v \pinv^* (\varphi \eta) dx} \\ & =\expect{\int_Q  (\pinv v) (\varphi \eta) dx},
\end{align*}
where we use the fact that $\unf^* P_{\mathrm{inv}}^*= P_{\mathrm{inv}}^*$ since the adjoint $P_{\mathrm{inv}}^*$ of $P_{\mathrm{inv}}$ satisfies $\mathrm{ran}(P_{\mathrm{inv}}^*)\subset L^q_{\mathrm{inv}}(\Omega)$. The claim follows by an approximation argument since $L^q(\Omega)\overset{a}{\otimes}L^q(Q)$ is dense in $L^q(\Omega\times Q)$.

\textit{Step 2. Convergence of $\pinv u\e$.} $\pinv$ is bounded and it commutes with $\nabla$, and therefore
\begin{equation*}
\limsup_{\varepsilon\to 0} \expect{\int_Q |\pinv u\e|^p+|\nabla \pinv u\e|^p dx}< \infty.
\end{equation*}
As a result of this and with help of Lemma \ref{lemma_basics} (ii) and Lemma \ref{lem6}, it follows that $\pinv u\e\overset{2}{\harpoon} v$ and $\nabla \pinv u\e\overset{2}{\harpoon} w$ (up to a subsequence), where $v\in L^p_{{\mathrm{inv}}}(\Omega)\otimes L^p(Q)$ and $w\in L^p_{{\mathrm{inv}}}(\Omega)\otimes L^p(Q)^d$.

Let $\varphi\in W^{1,q}(\Omega)$ and $\eta\in C^{\infty}_c(Q)$.
On the one hand, we have, as $\varepsilon\to 0$,
\begin{equation*}
\expect{\int_Q (\partial_i \pinv u\e) \mathcal{T}_{\varepsilon}^*(\varphi \eta) dx}=\expect{\int_Q\unf (\partial_i \pinv u\e) (\varphi\eta) dx}\rightarrow \expect{\int_Q w_i \varphi\eta dx}.
\end{equation*}
On the other hand, using $\partial_i\mathcal T\e^*(\varphi\eta)=\frac1\varepsilon\mathcal T\e^*(\eta D_i\varphi)+\mathcal T\e^*(\varphi\partial_i\eta)$ and $\mathcal T\e\pinv=\pinv$,
\begin{equation*}
\expect{\int_Q (\partial_i \pinv u\e) \mathcal{T}_{\varepsilon}^*(\varphi \eta) dx}=-\frac{1}{\varepsilon}\expect{\int_Q (\pinv u\e) (D_i\varphi\eta) dx}-\expect{\int_Q(\pinv u\e) \varphi \partial_i\eta dx}.
\end{equation*}
The first term on the right-hand side vanishes since $\pinv u\e(\cdot,x)\in L^p_{{\mathrm{inv}}}(\Omega)$ for almost every $x\in Q$ and by (\ref{orth}). The second term converges to $-\expect{\int_Q v \varphi \partial_i \eta dx}$ as $\varepsilon\rightarrow 0$. Consequently, we obtain $w=\nabla v$ and therefore $v\in L^p_{{\mathrm{inv}}}(\Omega)\otimes W^{1,p}(Q)$. Moreover, using Step 1, we have $\pinv u\e \harpoon u$ weakly in $L^p_{\mathrm{inv}}(\Omega)\otimes W^{1,p}(Q)$.

\textit{Step 3. Convergence of $u\e$.} Since $u\e$ is bounded, by  Lemma~\ref{lemma_basics}~(ii) and Lemma~\ref{lem6} there exists $u\in L^p_{{\mathrm{inv}}}(\Omega)\otimes L^p(Q)$ such that $u\e \overset{2}{\harpoon} u$ in $L^p(\Omega \times Q)$. Also, $\pinv$ is a linear and bounded operator which, together with Step 1, implies that $\pinv u\e\harpoon u$. Using this, we conclude that $u=v$.
\end{proof}
\begin{proof}[Proof of Proposition \ref{prop1}]
Lemma \ref{lem7} implies that $u\e \overset{2}{\harpoon} u$ in $L^p(\Omega \times Q)$ (up to a subsequence), where $u\in L^p_{{\mathrm{inv}}}(\Omega)\otimes W^{1,p}(Q)$. Moreover, it follows that there exists $v\in L^p(\Omega \times Q)^d$ such that $\nabla u\e \overset{2}{\harpoon} v$ in $L^p(\Omega \times Q)^d$ (up to another subsequence). We show that $\chi:=v-\nabla u\in L^p_{\mathrm{pot}}(\Omega)\otimes L^p(Q)$.

Let $\varphi \in W^{1,q}(\Omega)^d$ with $D^*\varphi=0$ and $\eta\in C^{\infty}_c(Q)$. We have, as $\varepsilon\to 0$,
\begin{equation}\label{eq4321}
\expect{\int_Q \nabla u\e \cdot \mathcal{T}_{\varepsilon}^*(\varphi \eta) dx} = \expect{\int_Q \unf \nabla u\e \cdot \varphi \eta dx} \rightarrow \expect{\int_{Q}v \cdot \varphi \eta dx}.
\end{equation}
On the other hand,
\begin{align}\label{eq1234}
\begin{split}
\expect{\int_Q \nabla u\e \cdot \mathcal{T}_{\varepsilon}^*(\varphi \eta) dx}  &=-\expect{\int_Q u\e \sum_{i=1}^d \mathcal{T}_{\varepsilon}^*(\frac{1}{\varepsilon}  \eta D_i \varphi_i+\varphi_i \partial_i\eta) dx}  
\\ &= \frac{1}{\varepsilon} \expect{\int_Q (\unf u\e) ( \eta D^*\varphi) dx}-\expect{\int_Q (\unf u\e) \sum_{i=1}^d \varphi_i\partial_i \eta dx}.
\end{split}
\end{align}
Above, the first term on the right-hand side vanishes by assumption and the second converges to $ \expect{\int_Q\nabla u\cdot \varphi \eta}$ as $\varepsilon\rightarrow 0$. Using
(\ref{eq1234}), (\ref{eq4321}) and Lemma \ref{lemA} (ii) we  complete the proof.
\end{proof}
\begin{proof}[Proof of Lemma \ref{Nonlinear_recovery}]
For $\chi \in L^p_{\mathrm{pot}}(\Omega)\otimes L^p(Q)$ and $\delta>0$, by definition of the space $L^p_{\mathrm{pot}}(\Omega)\otimes L^p(Q)$ and by density of $\mathrm{ran}(D)$ in $L^p_{\mathrm{pot}}(\Omega)$, we find $g_{\delta}=\sum_{i=1}^{n(\delta)}\varphi^{\delta}_i \eta^{\delta}_i$ with $\varphi_i^{\delta} \in W^{1,p}(\Omega)$ and $\eta^{\delta}_i\in C^{\infty}_c(Q)$ such that
\begin{equation*}
\|\chi - Dg_{\delta} \|_{L^p(\Omega\times Q)^d} \leq \delta.
\end{equation*}
Note that we can choose $\varphi_i^{\delta}$ above so that $\varphi_i^{\delta}\in L^{\theta}(\Omega)$. This can be seen by a standard truncation and mollification argument (see \cite[Lemma 2.2]{bourgeat1994stochastic} for the $L^2$-case) that we present here for the convenience of the reader. For a given $\varphi \in W^{1,p}(\Omega)$, by density of $L^{\infty}(\Omega)$ in $L^p(\Omega)$, we find a sequence $\varphi_{k}\in L^{\infty}(\Omega)$ such that $\varphi_k \to \varphi$ in $L^p(\Omega)$. For a sequence of standard mollifiers $\rho_n \in C_{c}^{\infty}(\R^d)$, $\rho_n \geq 0$, we define 
\begin{equation*}
\varphi_{k}^{n}= \int_{\R^d}\rho_n(y)U_{y}\varphi_k dy, \quad \varphi^{n}= \int_{\R^d}\rho_n(y)U_y\varphi dy.
\end{equation*} 
It holds $\varphi_{k}^n \in L^{\infty}(\Omega)\cap W^{1,p}(\Omega)$, $D_i\varphi_{k}^n =\int_{\R^d}-\partial_{i}\rho_n(y) U_y \varphi_k dy$ and $D_i \varphi^n = \int_{\R^d}-\partial_i \rho_n(y)U_y \varphi dy=\int_{\R^d} \rho_n(y)U_y D_i\varphi dy$. Similarly as in the proof of Lemma \ref{lemA} (ii), it follows that $D\varphi^n \to D\varphi$ in $L^p(\Omega)^d$ as $n\to \infty$. In the following we show that for fixed $n\in \N$, $D_i\varphi_k^n\to D_i\varphi^n$ in $L^p(\Omega)$ as $k\to \infty$, which yields the claim (up to extraction of a subsequence $k(n)$). We have, as $k\to \infty$,
\begin{equation*}
\expect{|D_i\varphi_k^n - D_i\varphi^n|^p}= \expect{\big| \int_{\R^d}-\partial_i \rho_n(y) \brac{U_y \varphi_k -U_y \varphi} dy \big|^p}\leq c(n) \expect{|\varphi_k-\varphi|^p}\to 0,
\end{equation*}
where in the last inequality we use that $\partial_i \rho_n$ is compactly supported and $L^{\infty}$, and Jensen's inequality. This means that in the definition of $g_{\delta}$ above, we can choose $\varphi_{i}^{\delta}\in L^{\theta}(\Omega)\cap W^{1,p}(\Omega)$.

We define $g_{\delta,\varepsilon}= \varepsilon \unf^{-1} g_{\delta}$ and note that $g_{\delta,\varepsilon} \in L^p(\Omega)\otimes W_0^{1,p}(Q)\cap L^{\theta}(\Omega\times Q)$ and $\nabla g_{\delta,\varepsilon}=\unf^{-1}D g_{\delta}+\unf^{-1}\varepsilon\nabla g_{\delta}$. As a result of this and with help of the isometry property of $\unf^{-1}$, the claim of the lemma follows.
\end{proof}
\section{Proof of Theorem \ref{s3_thm_5}}\label{Section_3.2}
Before presenting the main proof, we provide three auxiliary lemmas. Lemma \ref{lem:709:p} provides the reduction of the $\Lambda$-convex gradient flows to convex gradient flows. Lemmas \ref{lem:99:a2} and \ref{lem:128:a} provide a suitable recovery sequence that is helpful in the treatment of the term $\int_{0}^{T}\widetilde{\E}\e^*(t,-D\rcal\e(\dot{u}\e(t)))dt$ in \eqref{eq:509:p} (cf. \eqref{eq:738}). 
\begin{lemma}[Convex reduction]\label{lem:709:p}
Let the assumptions of Theorem \ref{s3_thm_5} be satisfied. Let $\widetilde{\E}\e:[0,T]\times Y \to \R\cup \cb{\infty}$ and $\widetilde{\E}\h:[0,T]\times Y_0\to \R\cup \cb{\infty}$ be given by
\begin{equation*}
\widetilde{\E}\e(t,u)= e^{2\Lambda t}\E\e(e^{-\Lambda t}u)-\Lambda \rcal\e(u), \quad \widetilde{\E}\h(t,u)= e^{2\Lambda t}\E\h(e^{-\Lambda t}u)-\Lambda \rcal\h(u). 
\end{equation*}
Then:
\begin{enumerate}[label=(\roman*)] 
\item $\widetilde{\E}\e$ and $\widetilde{\E}\h$ are convex normal integrands (see Definition \ref{def:appen}).
\item $y \in H^1(0,T; Y)$ satisfies \eqref{eq:317} if and only if $u(t):=e^{\Lambda t}y(t)$ satisfies 
\begin{equation}\label{eq:738}
\rcal\e(u(T))+\int_{0}^T \widetilde{\E}\e(t,u(t))+\widetilde{\E}^*\e(t,-D\rcal\e(\dot{u}(t)))dt = \rcal\e(u(0)),
\end{equation}
where $\widetilde{\E}\e^*(t,\cdot)$ denotes the convex conjugate of $\widetilde{\E}\e(t,\cdot)$.
\item $y\in H^1(0,T; Y_0)$ satisfies \eqref{eq:318} if and only if $u(t):=e^{\Lambda t} y(t)$ satisfies
\begin{equation*}
\rcal\h(u(T))+\int_{0}^T \widetilde{\E}\h(t,u(t))+\widetilde{\E}^*\h(t,-D\rcal\h(\dot{u}(t)))dt = \rcal\h(u(0)),
\end{equation*}
where $\widetilde{\E}\h^*(t,\cdot)$ denotes the convex conjugate of $\widetilde{\E}\h(t,\cdot)$.
\end{enumerate}
\end{lemma}
\begin{proof}
\textit{Proof of (i).}  For fixed $t$, convexity of $\widetilde{\E}\e(t,\cdot)$ follows from $\Lambda$-convexity of $\E\e$. $\widetilde{\E}\e(t,\cdot)$ is proper and l.s.c. Indeed, this follows by continuity of $\rcal\e$ and by the fact that $\E\e$ is proper and l.s.c.
In the following we show that $\widetilde{\E}\e$ is $\mathcal{L}(0,T)\otimes \mathcal{B}(Y)$-measurable that implies the claim for $\widetilde{\E}\e$. First, we note that $-\Lambda \rcal\e$ is $\mathcal{B}(Y)$-measurable since it is continuous, therefore it is sufficient to show that the mapping $(t,u)\mapsto e^{2\Lambda t}\E\e(e^{-\Lambda t}u)$ is $\mathcal{L}(0,T)\otimes \mathcal{B}(Y)$-measurable.
We note that $\E\e(e^{-\Lambda t}u)$ is the composition of the continuous mapping $(t,u)\mapsto e^{-\Lambda t}u$ (thus $\brac{\mathcal{B}(0,T)\otimes \mathcal{B}(Y),\mathcal{B}(Y)}$-measurable) and the l.s.c. functional $\E\e$ that is, thus, $\mathcal{B}(Y)$-measurable. As a result of this, it is $\mathcal{B}(0,T)\otimes \mathcal{B}(Y)$-measurable. Finally, the expression $e^{2 \Lambda t}\E\e(e^{-\Lambda t}u)$ is a product of a continuous and a measurable functional and therefore it is $\mathcal{L}(0,T)\otimes \mathcal{B}(Y)$-measurable. For $\widetilde{\E}\h$, the claim follows analogously.

\textit{Proof of  (ii).} Since $\rcal\e$ is quadratic we have  $\rcal\e(\widetilde y)=\frac12\expect{D\rcal\e(\widetilde y),\widetilde y}_{Y^*,Y}$. Combined  with \eqref{eq:317}, a simple rearrangement  yields for all $\widetilde{y}\in Y$,
\begin{equation*}
\expect{D\rcal\e\brac{\dot{y}(t)+\Lambda y(t)},y(t)-\widetilde{y}}_{Y^*,Y}+ \E\e(y(t))-\Lambda \rcal\e(y(t)) \leq \E\e(\widetilde{y})-\Lambda \rcal\e(\widetilde{y}).
\end{equation*}
We multiply the above inequality with $e^{2\Lambda t}$ and use linearity of $D\rcal\e$ (resp. quadratic structure of $\rcal\e$) to obtain,
\begin{eqnarray*}
  & & \expect{D\rcal\e \brac{e^{\Lambda t}\dot{y}(t)+\Lambda e^{\Lambda t}y(t)},e^{\Lambda t}(y(t)-\widetilde{y})}_{Y^*,Y}\\ & & + e^{2\Lambda t}\E\e(e^{-\Lambda t}e^{\Lambda t}y(t))-\Lambda \rcal\e(e^{\Lambda t}y(t)) \\ & \leq & e^{2\Lambda t} \E\e(e^{-\Lambda t}e^{\Lambda t}\widetilde{y}) -\Lambda \rcal\e(e^{\Lambda t}\widetilde y) \quad \text{for all }\widetilde y \in Y.
\end{eqnarray*}
With $u(t)=e^{\Lambda t}y(t)$, the definition of $\widetilde{\E}\e$, and with the test-function $\widetilde y=e^{-\Lambda t}\hat y$, the above
inequality reads
\begin{equation*}
\expect{D\rcal\e(\dot{u}(t)),u(t)-\hat{y}}_{Y^*,Y}+\widetilde{\E}\e(t,u(t))\leq \widetilde{\E}\e(t,\hat{y}) \quad \text{for all }\hat{y}\in Y,
\end{equation*}
where we used that $\dot{u}(t)=e^{\Lambda t}\dot{y}(t)+\Lambda e^{\Lambda t} y(t)$. Since $\widetilde{\E}\e(t,\cdot)$ is convex for each $t$, the Fenchel equivalence implies that $u$ satisfies for a.e.~$t\in(0,T)$,
\begin{equation}\label{eq:767}
  \expect{D\rcal\e(\dot{u}(t)),u(t)}_{Y^*,Y}+\widetilde{\E}\e(t,u(t))+\widetilde{\E}^*\e(t,-D\rcal\e( \dot{u}(t)))=0.
\end{equation}
Since $\frac{d}{dt}\rcal\e(u(t))=\expect{D\rcal\e(u(t)),\dot{u}(t)}_{Y^*,Y}=\expect{D\rcal\e(\dot{u}(t)),u(t)}_{Y^*,Y}$, integration of the above identity over 
$(0,T)$ yields \eqref{eq:738}. On the other hand, if \eqref{eq:738} holds, then we have
\begin{equation*}
\int_{0}^T\expect{D\rcal\e(\dot{u}(t)),u(t)}_{Y^*,Y} +\widetilde{\E}\e(t,u(t))+\widetilde{\E}^*\e(t,-D\rcal\e(\dot{u}(t)))dt = 0.
\end{equation*} 
The integrand on the left-hand side is nonnegative by the definition of the convex conjugate and therefore it follows that $u$ satisfies \eqref{eq:767}. This completes the proof.

\textit{Proof of (iii).} The argument is the same as in part (ii).
\end{proof}
\begin{remark}[Extended unfolding] For $p\in (1,\infty)$, the stochastic unfolding operator $\unf: L^p(\Omega\times Q)\to L^p(\Omega\times Q)$ can be extended to a (not relabeled) linear isometry $\unf: L^p(0,T;L^p(\Omega\times Q))\to L^p(0,T; L^p(\Omega\times Q))$. In particular, for functions of the form $u=\eta \varphi\in L^p(0,T;L^p(\Omega\times Q))$ with $\eta \in L^p(0,T)$ and $\varphi \in L^p(\Omega\times Q)$, we define the unfolding by 
\begin{equation*}
\unf u(t,\cdot)=\eta(t)\unf \varphi(\cdot).
\end{equation*}   
By the density of $\cb{\sum_{i}\eta_i\varphi_i:\; \eta_i \in L^p(0,T),\; \varphi_i \in L^p(\Omega\times Q)}$ in $L^p(0,T;L^p(\Omega\times Q))$ we may extend the unfolding operator to a uniquely determined isometry on $L^p(0,T; L^p(\Omega\times Q))$. In the following, we use this extension.
\end{remark}

\begin{lemma}[Recovery sequence]\label{lem:99:a2}
Let $p\in (1,\infty)$, $\theta \in [2,\infty)$ and $Q\subset \R^d$ be open and bounded. Let $w\in L^p(0,T; L^p_{\mathrm{inv}}(\Omega)\otimes W^{1,p}_0(Q)) \cap L^{\theta}(0,T;L^{\theta}_{\mathrm{inv}}(\Omega)\otimes L^{\theta}(Q))$ and $\chi \in L^p(0,T; L^p\p(\Omega)\otimes L^p(Q))$. Then, there exists $w\e \in L^p(0,T; L^p(\Omega)\otimes W_0^{1,p}(Q)) \cap L^{\theta}(0,T; L^{\theta}(\Omega\times Q))$ such that, as $\varepsilon\to 0$,
\begin{align*}
& \unf w\e \to w \quad \text{strongly in }L^{\theta}(0,T; L^{\theta}(\Omega\times Q)),\\
& \unf \nabla w\e \to \nabla w+\chi \quad \text{strongly in }L^p(0,T; L^p(\Omega\times Q)^d).
\end{align*}
\end{lemma}
\begin{proof}
Since $\chi \in  L^p(0,T; L^p_{\mathrm{pot}}(\Omega)\otimes L^p(Q))$, we find a sequence $\psi^{k}=\sum_{i=1}^{k}\eta^{k,i} \chi^{k,i}$ with $\eta^{k,i}\in C^{\infty}_c(0,T)$ and $\chi^{k,i} \in L^p_{\mathrm{pot}}(\Omega)\otimes L^p(Q)$, such that
\begin{equation*}
\|\psi^k-\chi\|_{L^p(0,T;L^p(\Omega\times Q)^d)}\to 0 \quad \text{as }k \to \infty.
\end{equation*}
In view of Lemma \ref{Nonlinear_recovery}, for each $\chi^{k,i}$ we find $g_{\delta, \varepsilon}^{k,i}\in (L^p(\Omega)\otimes W^{1,p}_0(Q))\cap L^{\theta}(\Omega\times Q)$ such that
\begin{equation*}
\|g_{\delta,\varepsilon}^{k,i}\|_{L^{\theta}(\Omega\times Q)} \leq \varepsilon c_{k,i}(\delta), \quad \limsup_{\varepsilon\to 0}\|\mathcal{T}_{\varepsilon}\nabla g_{\delta, \varepsilon}^{k,i}-\chi^{k,i}\|_{L^p(\Omega\times Q)^d}\leq \delta.
\end{equation*}

We define $w_{\delta,\varepsilon}^k= w+ \sum_{i=1}^{k}\eta^{k,i} g_{\delta,\varepsilon}^{k,i}$ and we estimate
\begin{eqnarray*}
& & \|\unf w_{\delta,\varepsilon}^k-w\|_{L^{\theta}(0,T; L^{\theta}(\Omega\times Q))}+\|\unf \nabla w_{\delta,\varepsilon}^k - (\nabla w + \chi)\|_{L^p(0,T; L^p(\Omega\times Q)^d)}\\ & \leq & \|\sum_{i=1}^{k}\eta^{k,i} g_{\delta
,\varepsilon}^{k,i}\|_{L^{\theta}(0,T; L^{\theta}(\Omega\times Q))}+ \|\sum_{i=1}^{k}\eta^{k,i} \brac{\unf \nabla g_{\delta,\varepsilon}^{k,i} - \chi^{k,i}}\|_{L^p(0,T; L^p(\Omega\times Q)^d)} \\
& & + \norm{\psi^k-\chi}_{L^p(0,T; L^p(\Omega\times Q)^d)}\\
& \leq & \varepsilon \sum_{i=1}^{k}c_{k,i}(\delta) + \sum_{i=1}^{k} c_{k,i} \norm{\unf \nabla g_{\delta,\varepsilon}^{k,i} - \chi^{k,i}}_{L^p(\Omega\times Q)^d} + \norm{\psi^k-\chi}_{L^p(0,T; L^p(\Omega\times Q)^d)}.
\end{eqnarray*}
Letting first $\varepsilon \to 0$, secondly $\delta \to 0$, and finally $k\to \infty$, the right-hand side above vanishes. As a result of this, we can extract diagonal sequences $k(\varepsilon)$ and $\delta(\varepsilon)$ such that $w\e:=w^{k(\varepsilon)}_{\delta(\varepsilon),\varepsilon}$ satisfies the claim of the lemma.
\end{proof}

\begin{lemma}[Measurable selection]\label{lem:128:a}
Let the assumptions of Lemma \ref{lem:709:p} be satisfied. Let $\xi \in  L^2(0,T; Y_0^*)$. There exists $w \in L^p(0,T;L^p_{\mathrm{inv}}(\Omega)\otimes W^{1,p}_0(Q))\cap L^{\theta}(0,T; L^{\theta}_{\mathrm{inv}}(\Omega)\otimes L^{\theta}(Q))$ such that
\begin{equation*}
\int_{0}^T \widetilde{\E}\h^*(t,\xi(t))dt = \int_{0}^{T}\expect{\xi(t),w(t)}_{Y_0^*,Y_0}dt-\int_{0}^{T}\widetilde{\E}\h(t,w(t))dt.
\end{equation*}
Moreover, there exists $\chi\in L^p(0,T; L^p\p(\Omega)\otimes L^p(Q))$ such that
\begin{eqnarray}\label{eq:745}
& & \int_{0}^T \inf_{\chi\in L^p\p(\Omega)\otimes L^p(Q)} \expect{\int_{Q}e^{2 \Lambda t}V(\omega,x,e^{-\Lambda t}\nabla w(t)+\chi)dx} dt \nonumber \\
& = & \int_{0}^T \expect{\int_{Q}e^{2 \Lambda t}V(\omega,x,e^{-\Lambda t}\nabla w(t)+\chi(t))dx}dt. 
\end{eqnarray}
\end{lemma}
\begin{proof}
First we note that $\widetilde{\E}_{\mathrm{hom}}$ is a convex normal integrand by Lemma \ref{lem:709:p} (i) and $\int_{0}^T\widetilde{\E}\h(t,0)dt<\infty$. Therefore, Proposition \ref{prop:appen1} in Appendix~\ref{sec:norm:a} implies that 
\begin{equation}\label{eq:713}
  \begin{aligned}
    &\int_{0}^T \widetilde{\E}_{\mathrm{hom}}^*(t,\xi(t))dt \\
    &=
    \sup_{w\in L^2(0,T;
      Y_0)}\left(\int_{0}^T\expect{\xi(t),w(t)}_{Y_0^*,Y_0}dt-\int_{0}^T
      \widetilde{\E}_{\mathrm{hom}}(t,w(t))dt\right).
  \end{aligned}
\end{equation}
Using the direct method of the calculus of variations, with the help of the growth conditions of $V$ and $f$, we conclude that the supremum on the right-hand side is attained by some $w\in L^2(0,T; Y_0)$. As a result of this, we have $\int_{0}^T\widetilde{\E}\h(t,w(t))dt < \infty$, which implies that $w \in L^p(0,T;L^p_{\mathrm{inv}}(\Omega)\otimes W^{1,p}_0(Q))\cap L^{\theta}(0,T; L^{\theta}_{\mathrm{inv}}(\Omega)\otimes L^{\theta}(Q))$.

To show \eqref{eq:745}, we define an integrand $\mathcal{I}:[0,T]\times \brac{L^p_{\mathrm{pot}}(\Omega)\otimes L^p(Q)}\to \R\cup \cb{\infty}$ by $\mathcal{I}(t,\chi)=e^{2\Lambda t}\expect{\int_{Q}V(\omega,x,e^{-\Lambda t}\nabla w(t)(\omega,x)+\chi(\omega,x)dx}$. We remark that $\mathcal{I}$ is finite everywhere (up to considering a suitable representative of $\nabla w$) and for all $t\in [0,T]$, $\mathcal{I}(t,\cdot)$ is convex and l.s.c. (using the growth conditions of $V$), in fact, $\mathcal{I}(t,\cdot)$ is continuous. Moreover, for each fixed $\chi \in L^p\p(\Omega)\otimes L^p(Q)$, $\mathcal{I}(\cdot,\chi)$ is $\mathcal{L}(0,T)$-measurable. Indeed, this follows by the observation that $\mathcal{I}(\cdot,\chi)$ is a composition of the mappings $g_1:[0,T]\to [0,T]\times L^{p}(\Omega\times Q)^d$, $g_1(t)=\brac{t,e^{-\Lambda t}\nabla w(t)+\chi}$, and $g_2:[0,T]\times L^{p}(\Omega\times Q)^d\to \R$, $g_2(t,\varphi)=e^{2\Lambda t}\expect{\int_{Q}V(\omega,x,\varphi(\omega,x))dx}$. $g_1$ is $\brac{\mathcal{L}(0,T), \mathcal{L}(0,T)\otimes \mathcal{B}(L^p(\Omega\times Q)^d)}$-measurable and $g_2$ is a Carath{\'e}odory integrand and thus $\brac{\mathcal{L}(0,T)\otimes \mathcal{B}(L^p(\Omega\times Q)^d)}$-measurable. The above statements imply that $\mathcal{I}$ is a convex Carath{\'e}odory integrand, thus a normal convex integrand (see Appendix \ref{sec:norm:a}). As a result of this, Proposition \ref{prop:appen1} (and in particular Remark \ref{rem:901:a}) in Appendix~\ref{sec:norm:a} implies that
\begin{equation*}
\int_{0}^T\inf_{\chi\in L^p\p(\Omega)\otimes L^p(Q)}\mathcal{I}(t,\chi)dt = \inf_{\chi\in L^p(0,T; L^p\p(\Omega)\otimes L^p(Q))}\int_{0}^T\mathcal{I}(t,\chi(t))dt.
\end{equation*}
The infimum on the right-hand side is attained at some $\chi\in L^p(0,T; L^p\p(\Omega)\otimes L^p(Q))$, using the direct method of the calculus of variations. This concludes the proof.
\end{proof}

\begin{proof}[Proof of Theorem \ref{s3_thm_5}] 

\textit{Step 1. Compactness.}
The apriori estimate \eqref{eq:295} and the boundedness of $\E\e(y^0\e)$ yield, for all $t\in [0,T]$,
\begin{equation}\label{eq:102:a2}
\norm{y\e(t)}^p_{L^p(\Omega)\otimes W^{1,p}(Q)}+\norm{y\e(t)}^{\theta}_{L^{\theta}(\Omega\times Q)} \leq c. 
\end{equation}
Also, by the isometry property of $\unf$ and since $\theta\geq 2$, the above implies that $\norm{\unf y\e(t)}_{Y}^{\theta}\leq c$. We remark that $\unf y\e \in H^1(0,T;Y)$ since $\dot{\brac{\cdot}}$ and $\unf$ commute, i.e., $\frac{d}{dt}\brac{\unf y\e}=\unf \dot{y}\e$, where on the left-hand side $\unf y\e$ is pointwise defined as $\unf y\e(t)$ and on the right-hand side $\unf$ is the extension defined on $L^2(0,T; Y)$. As a result of this and using the isometry property of $\unf$, the apriori estimate \eqref{eq:323} implies that 
\begin{equation*}
\norm{\unf y\e}_{H^1(0,T;Y)}^2 \leq c, \quad \norm{\unf y\e(t)-\unf y\e(s)}^2_{Y}\leq c |t-s| \quad \text{for all }s,t \in [0,T].
\end{equation*}
We extract a (not relabeled) subsequence and $y\in H^1(0,T; Y)$ such that $\unf y\e \harpoon y$ weakly in $H^1(0,T; Y)$, and this implies that $\unf \dot{y}\e \harpoon \dot{y}$ weakly in $L^2(0,T; Y)$. We apply the Arzel{\`a}-Ascoli theorem to the sequence $\unf y\e$ to obtain that (up to another subsequence) for all $t\in [0,T]$,
\begin{equation}\label{eq:209:a2}
\unf y\e(t) \harpoon y(t) \quad \text{weakly in }Y. 
\end{equation}
Using (\ref{eq:102:a2}) and Proposition \ref{prop1}, we conclude that $y(t)\in (L^p_{\mathrm{inv}}(\Omega)\otimes W^{1,p}_0(Q))\cap \brac{L^{\theta}_{\mathrm{inv}}(\Omega)\otimes L^{\theta}(Q)}$ and $\unf y\e(t)\harpoon y(t)$ weakly in $L^{\theta}(\Omega\times Q)$ and in $L^p(\Omega\times Q)$ (see also Remark \ref{R_Two_1}). This also implies that $y\in H^1(0,T; Y_0)$. Moreover, for each $t\in [0,T]$ we find $\chi(t)\in L^p_{\mathrm{pot}}(\Omega)\otimes L^p(Q)$ and a subsequence $\varepsilon(t)$ such that $\mathcal{T}_{\varepsilon(t)}\nabla y_{\varepsilon(t)}(t) \harpoon \nabla y(t)+\chi(t)$ weakly in $L^p(\Omega\times Q)^d$. This implies that $P_{\mathrm{inv}}\nabla y\e(t) \harpoon \nabla y(t)$ weakly in $L^p(\Omega\times Q)^d$ for the whole (sub)sequence $\varepsilon$. Note that the assumption on the initial data implies that $\unf y\e(0)\to y^0$ strongly in $Y$ and hence we have $y(0)=y^0$.

In the following step, using Lemma \ref{lem:709:p}, we restate \eqref{eq:317} as a convex problem. For this reason, we define the new variables $u\e(t)=e^{\Lambda t}y\e(t)$ and $u(t)=e^{\Lambda t}y(t)$. Note that $\dot{u}\e(t)=\Lambda e^{\Lambda t} y\e(t)+e^{\Lambda t}\dot{y}\e(t)$ and analogously for $\dot{u}$. The above convergence statements result in 
\begin{align}\label{eq:117:a2}
\begin{split}
& \unf u\e \harpoon u \quad \text{weakly in }H^1(0,T; Y),\\
& \unf u\e(t) \harpoon u(t) \quad \text{weakly in }L^{\theta}(\Omega\times Q) \text{ and }L^p(\Omega\times Q), \quad \text{for all }t \in [0,T].
\end{split}
\end{align}

\textit{Step 2. Reduction to a convex problem.} In view of Lemma \ref{lem:709:p} (ii), we have
\begin{equation}\label{eq:129:a2}
\rcal\e(u\e(T))+ \int_{0}^T \widetilde{\E}\e(t,u\e(t))+\widetilde{\E}\e^*(t,-D\rcal\e(\dot{u}\e(t)))dt = \rcal\e(u\e(0)).
\end{equation}

\textit{Step 3. Passage to the limit $\varepsilon\to 0$ in (\ref{eq:129:a2}).} Note that $u\e(0) =y\e^0\overset{2}{\to} y^0 = u(0)$ in $Y$ and therefore using Proposition \ref{P_Cont_1} (ii), for the right-hand side of (\ref{eq:129:a2}), we have
\begin{equation}\label{eq:233:2a}
\lim_{\varepsilon\to 0}\rcal\e(u\e(0))=\rcal\h(u(0)).
\end{equation}

The first term on the left-hand side is treated similarly, using Proposition \ref{P_Cont_1} (iii) and (\ref{eq:117:a2}), we have 
\begin{equation}\label{eq:237:2a}
\liminf_{\varepsilon\to 0}\rcal\e(u\e(T))\geq \rcal\h(u(T)).
\end{equation}

We treat the second term on the left-hand side of (\ref{eq:129:a2}) as follows. By Fatou's lemma we have 
\begin{eqnarray*}
&& \liminf_{\varepsilon\to 0}\int_{0}^T\widetilde{\E}\e(t,u\e(t)) dt \\  & \geq &  \int_0^T \liminf_{\varepsilon\to 0}\expect{\int_{Q}e^{2 \Lambda t}V(\tau_{\frac{x}{\varepsilon}}\omega,x,e^{-\Lambda t}\nabla u\e(t))dx}dt\\ & & +\int_{0}^T \liminf_{\varepsilon\to 0}\expect{\int_{Q}e^{2\Lambda t}f(\tau_{\frac{x}{\varepsilon}}\omega,x,e^{-\Lambda t}u\e(t))-\frac{\Lambda}{2}r(\tau_{\frac{x}{\varepsilon}}\omega,x)|u\e(t)|^2 dx}dt.
\end{eqnarray*}
For fixed $t$, the $\liminf$ in the first term is a limit for a subsequence $\varepsilon(t)$ and as in Step 1 we find $\chi(t)\in L^p\p(\Omega)\otimes L^p(Q)$ such that, up to another (not relabeled) subsequence, it holds $\nabla u_{\varepsilon(t)}(t) \overset{2}{\harpoon} \nabla u(t)+e^{\Lambda t}\chi(t)$ in $L^p(\Omega\times Q)^d$. Also, we notice that $e^{2\Lambda t}V(\omega,x,e^{-\Lambda t}\cdot)$ is convex and has $p$-growth properties and therefore Proposition~\ref{P_Cont_1}~(iii) implies that
\begin{eqnarray*}
& & \liminf_{\varepsilon \to 0}\expect{\int_{Q}e^{2 \Lambda t}V(\tau_{\frac{x}{\varepsilon}}\omega,x,e^{-\Lambda t}\nabla u\e(t))dx} \\ & \geq & \expect{\int_{Q}e^{2 \Lambda t}V(\omega,x,e^{-\Lambda t}\nabla u(t)+\chi(t))dx}\\
& \geq & \inf_{\chi\in L^p\p(\Omega)\otimes L^p(Q)}\expect{\int_{Q}e^{2\Lambda t}V(\omega,x,e^{-\Lambda t}\nabla u(t)+\chi)dx}. 
\end{eqnarray*}
On the other hand, we remark that the integrand $e^{2\Lambda t}f(\omega,x,e^{-\Lambda t}\cdot)-\frac{\Lambda}{2}r(\omega,x)|\cdot|^2$ is convex and satisfies $\theta$-growth conditions. As a result of this and by (\ref{eq:117:a2}), Proposition~\ref{P_Cont_1}~(iii) yields
\begin{eqnarray*}
& & \liminf_{\varepsilon\to 0}\expect{\int_{Q}e^{2\Lambda t}f(\tau_{\frac{x}{\varepsilon}}\omega,x,e^{-\Lambda t}u\e(t))-\frac{\Lambda}{2}r(\tau_{\frac{x}{\varepsilon}}\omega,x)|u\e(t)|^2 dx} \\ & \geq & \expect{\int_{Q}e^{2\Lambda t}f(\omega,x,e^{-\Lambda t}u(t))-\frac{\Lambda}{2}r(\omega,x)|u(t)|^2 dx}.
\end{eqnarray*}
Using the above two statements we conclude that
\begin{equation}\label{eq:254:2a}
\liminf_{\varepsilon\to 0}\int_{0}^T\widetilde{\E}\e(t,u\e(t))dt \geq \int_{0}^T \widetilde{\E}\h(t,u(t))dt. 
\end{equation}

In order to complete the limit passage, it is left to treat the third term on the left-hand side of (\ref{eq:129:a2}). Using Lemma \ref{lem:128:a}, we find $w \in L^p(0,T;L^p_{\mathrm{inv}}(\Omega)\otimes W^{1,p}_0(Q))\cap L^{\theta}(0,T; L^{\theta}_{\mathrm{inv}}(\Omega)\otimes L^{\theta}(Q))$ such that
\begin{equation*}
\int_{0}^T \widetilde{\E}\h^*(t,-D\rcal\h(\dot{u}(t)))dt = \int_{0}^{T}\expect{-D\rcal\h(\dot{u}(t)),w(t)}_{Y_0^*,Y_0}-\widetilde{\E}\h(t,w(t))dt.
\end{equation*}
Moreover, by the second claim of Lemma \ref{lem:128:a}, we find $\chi\in L^p(0,T; L^p\p(\Omega)\otimes L^p(Q))$ such that
\begin{align}\label{eq:536:ac2}
\int_0^T\widetilde{\E}\h(t,w(t))dt = & \int_{0}^T e^{2 \Lambda t}\expect{\int_{Q}V(\omega,x,e^{-\Lambda t}\nabla w(t)+\chi(t))+f(\omega,x,e^{-\Lambda t}w(t))} \nonumber \\ & -\Lambda \rcal\h(w(t))dt. 
\end{align}
For the pair $\brac{w,e^{\Lambda \cdot}\chi(\cdot)}$ ($e^{\Lambda\cdot}$ denotes the function $t\mapsto e^{\Lambda t}$) Lemma \ref{lem:99:a2} implies the existence of $w\e \in L^p(0,T; L^p(\Omega)\otimes W^{1,p}_0(Q)) \cap L^{\theta}(0,T; L^{\theta}(\Omega\times Q))$ such that
\begin{align}\label{eq:270:a2}
\begin{split}
& \unf w\e \to w \quad \text{strongly in } L^{\theta}(0,T; L^{\theta}(\Omega\times Q)),\\
& \unf \nabla w\e \to \nabla w+e^{\Lambda \cdot}\chi \quad \text{strongly in }L^p(0,T; L^p(\Omega\times Q)^d).
\end{split}
\end{align}
Using the definition of the convex conjugate $\widetilde{\E}\e^*$, we have
\begin{equation*}
\int_{0}^T \widetilde{\E}\e^*(t,-D\rcal\e(\dot{u}\e(t)))dt \geq \int_{0}^T \expect{-D\rcal\e(\dot{u}\e(t)),w\e(t)}_{Y^*,Y}-\widetilde{\E}\e(t,w\e(t))dt.
\end{equation*}
For the first term on the right-hand side we have, using the fact that the extended unfolding operator is unitary, as $\varepsilon\to 0$,
\begin{align}
& \int_{0}^T\expect{-D\rcal\e(\dot{u}\e(t)),w\e(t)}_{Y^*,Y}dt=-\int_{0}^T \expect{\int_{Q}r(\omega,x)\unf \dot{u}\e(t) \unf w\e(t)dx}dt \label{term:280}\\ & \to -\int_0^T \expect{\int_{Q}r(\omega,x)\dot{u}(t)w(t)dx}dt=\int_{0}^T \expect{-D\rcal\h(\dot{u}(t)),w(t)}_{Y^*_0,Y_0}dt.\nonumber
\end{align}
The above convergence follows since (\ref{term:280}) is a scalar product of a strongly and a weakly convergent sequence. Moreover, by Proposition \ref{P_Cont_1} (i),
\begin{eqnarray*}
& & \int_{0}^T \widetilde{\E}\e(t,w\e(t))dt \\ & = & \int_{0}^T e^{2\Lambda t}\expect{\int_{Q}V(\omega,x,e^{-\Lambda t}\unf \nabla w\e(t))+f(\omega,x,e^{-\Lambda t}\unf w\e(t))dx} dt \\ & & - \int_0^T \expect{ \int_{Q} \frac{\Lambda r}{2 }|\unf w\e(t)|^2 dx}dt.
\end{eqnarray*}
As $\varepsilon \to 0$, this expression converges to 
\begin{equation*}
\int_{0}^T e^{2\Lambda t}\expect{\int_{Q}V(\omega,x,e^{-\Lambda t}\nabla w(t)+\chi(t))+f(\omega,x,e^{-\Lambda t}w(t))- \frac{\Lambda r}{2 e^{2\Lambda t}}|w(t)|^2 dx}dt.
\end{equation*}
This follows completely analogously as in the proof of Proposition \ref{P_Cont_1} (ii) using the strong convergences (\ref{eq:270:a2}) and the growth conditions of the integrands (standard argument using Fatou's lemma). By (\ref{eq:536:ac2}), the last expression equals $\int_0^T\widetilde{\E}\h(t,w(t))dt$ and therefore collecting the above statements we conclude that
\begin{equation}\label{eq:292:a2}
\liminf_{\varepsilon\to 0}\int_{0}^T \widetilde{\E}\e^*(t,-D\rcal\e(\dot{u}\e(t)))dt \geq \int_{0}^T \widetilde{\E}\h^*(t,-D\rcal\h(\dot{u}(t)))dt.
\end{equation}

Collecting (\ref{eq:233:2a}), (\ref{eq:237:2a}), (\ref{eq:254:2a}) and (\ref{eq:292:a2}), we obtain that  
\begin{eqnarray*}
& & \int_{0}^T \widetilde{\E}_{\mathrm{hom}}(t,u(t))+\widetilde{\E}^*_{\mathrm{hom}}(t,-D\mathcal{R}_{\mathrm{hom}}(\dot{u}(t))) dt\\ & \leq & -\mathcal{R}_{\mathrm{hom}}(u(T))+\mathcal{R}_{\mathrm{hom}}(u(0))=\int_{0}^T \expect{-D\mathcal{R}_{\mathrm{hom}}(\dot{u}(t)),u(t)}_{Y_0^*,Y_0}dt.
\end{eqnarray*}
This inequality is, in fact, an equality by the Fenchel-Young inequality. Since $u(t)=e^{\Lambda t}y(t)$, Lemma \ref{lem:709:p} (iii) implies that $y$ is the unique solution to \eqref{eq:318} with $y(0)=y^0$.
Furthermore, using (\ref{eq:233:2a}) and (\ref{eq:237:2a}) we obtain
\begin{equation*}
\limsup_{\varepsilon\to 0} \brac{-\mathcal{R}\e(u\e(T))+\mathcal{R}\e(u\e(0))}\leq -\mathcal{R}_{\mathrm{hom}}(u(T))+\mathcal{R}_{\mathrm{hom}}(u(0)).
\end{equation*}
Also, exploiting the equality (\ref{eq:129:a2}) and the liminf inequalities (\ref{eq:254:2a}), (\ref{eq:292:a2}), we obtain
\begin{align*}
\liminf_{\varepsilon\to 0}\brac{-\mathcal{R}\e(u\e(T))+\mathcal{R}\e(u\e(0))} & \geq \int_{0}^T \widetilde{\E}_{\mathrm{hom}}(t,u(t))+\widetilde{\E}^*_{\mathrm{hom}}(t,-D\mathcal{R}_{\mathrm{hom}}(\dot{u}(t))) dt\\
& = - \rcal\h(u(T))+\rcal\h(u(0)).
\end{align*} 
This results in
\begin{equation*}
\lim_{\varepsilon\to 0}\frac{e^{2\Lambda T}}{2}\expect{\int_{Q}r(\omega,x)|\unf y\e(T)|^2 dx}= \lim_{\varepsilon\to 0} \rcal\e(u\e(T)) =\rcal\h(u(T)),
\end{equation*}
where we use that $\rcal\e(u\e(0))$ converges to $\rcal\h(u(0))$. Moreover, we note that $\rcal\h(u(T))=\frac{e^{2\Lambda T}}{2}\expect{\int_{Q}r(\omega,x)|y(T)|^2 dx}$; therefore, the above and (\ref{eq:209:a2}) imply that $\unf y\e(T)\to y(T)$ strongly in $Y$. Since $\unf y(T)=y(T)$ by shift-invariance of $y(T)$, we obtain that $y\e(T)\to y(T)$ strongly in $Y$. We may replace $T$ by any $t\in (0,T]$ in the above procedure to obtain $y\e(t)\to y(t)$ strongly in $Y$. Convergence for the entire sequence is obtained by a standard contradiction argument using the uniqueness of the solution for the limit problem.

\textit{Step 4. Convergence of $\dot{y}\e$ and $\E\e(y\e(t))$.} The EVI \eqref{eq:317} is equivalent to the differential inclusion (cf. \eqref{eq:145} in the Introduction)
\begin{equation*}
0 \in D\rcal\e(\dot{y}\e(t))+ \partial_{F}\E\e(y\e(t)) \quad \text{for a.e. } t \in (0,T).
\end{equation*}
This and the chain rule for the $\Lambda$-convex functional $\E\e$ (see, e.g., \cite{rossi2006gradient}) imply that $\frac{d}{dt}\E\e(y\e(t))=-\expect{D\rcal\e(\dot{y}\e(t)),\dot{y}\e}_{Y^*,Y}$. An integration over $(0,t)$, for an arbitrary $t\in (0,T]$, yields
\begin{equation*}
\int_{0}^t \expect{D\mathcal{R}\e(\dot{y}\e(s)),\dot{y}\e(s)}_{Y^{*},Y}ds = \E\e(y\e(0))-\E\e(y\e(t)). 
\end{equation*}
Since $y\e(t)\to y(t)$ strongly in $Y$ and by (\ref{eq:117:a2}), we obtain that $\liminf_{\varepsilon\to 0}\E\e(y\e(t))\geq \E_{\mathrm{hom}}(y(t))$, which follows using Proposition \ref{P_Cont_1} (cf. \eqref{eq:254:2a}). As a consequence, using the additional assumption $\E\e(y\e(0))\to \E_{\mathrm{hom}}(y(0))$, we obtain 
\begin{align*}
\limsup_{\varepsilon\to 0}\int_{0}^t \expect{D\mathcal{R}\e(\dot{y}\e(s)),\dot{y}\e(s)}_{Y^{*},Y}ds & \leq \E_{\mathrm{hom}}(y(0))-\E_{\mathrm{hom}}(y(t)) \\ & = \int_{0}^t \expect{D\mathcal{R}_{\mathrm{hom}}(\dot{y}(s)),\dot{y}(s)}_{Y^{*}_0,Y_0}ds, 
\end{align*}
where in the last equality we use that $y$ is the solution to the limit problem. Note that it holds $\int_{0}^t\expect{D\mathcal{R}\e(\dot{y}\e(s)),\dot{y}\e(s)}_{Y^{*},Y}ds = \int_0^t\expect{\int_{Q} r |\unf \dot{y}\e(s)|^2 dx}ds$ and since $\unf \dot{y}\e \rightharpoonup \dot{y}$ weakly in $L^2(0,T; Y)$, it follows that
\begin{equation*}
\liminf_{\varepsilon\to 0} \int_{0}^t\expect{D\mathcal{R}\e(\dot{y}\e(s)),\dot{y}\e(s)}_{Y^{*},Y}ds \geq \int_{0}^t \expect{D\mathcal{R}_{\mathrm{hom}}(\dot{y}(s)),\dot{y}(s)}_{Y^{*}_0,Y_0}ds.
\end{equation*}
Combining the last two inequalities and the weak convergence $\unf \dot{y}\e \rightharpoonup \dot{y}$, we conclude that for all $t\in(0,T]$,
\begin{equation*}
\dot{y}\e \to \dot{y} \quad \text{strongly in }L^2(0,t; Y), \quad \E\e(y\e(t))\to \E_{\mathrm{hom}}(y(t)).
\end{equation*} 
\end{proof}
\appendix
\section{Normal integrands and integral functionals}\label{sec:norm:a}
In the following we recall some key facts about measurable integrands and conjugates of integral functionals. A detailed and more general theory can be found in \cite{rockafellar1971convex}.

Let $(S, \Sigma, \mu)$ be a complete measure space with a $\sigma$-finite measure $\mu$ and let $X$ be a separable reflexive Banach space with dual space $X^*$. The product-$\sigma$-algebra of $\Sigma$ and $\mathcal{B}(X)$ (Borel $\sigma$-algebra on $X$) is denoted by $\Sigma\otimes \mathcal{B}(X)$. In the following we refer to a function $f: S \times X \to \R\cup \cb{\infty}$ as an \textit{integrand}. For $s\in S$, we denote the function $x\mapsto f(s,x)$ by $f_s$.
\begin{definition}[Normal integrand]\label{def:appen}
We say that an integrand $f$ is \textit{normal} if the following two conditions hold:
\begin{enumerate}[label=(\roman*)] 
\item $f$ is $\Sigma \otimes \mathcal B(X)$-measurable.
\item For each $s \in S$, the function $f_s$ is proper and l.s.c.
\end{enumerate}
If additionally, for each $s\in S$, $f_s$ is convex, we say that $f$ is a \textit{convex normal integrand}.
\end{definition}
Note that if $f$ is a normal integrand and $x:S\to X$ is a $\brac{\Sigma,\mathcal{B}(X)}$-measurable function, then $s\mapsto f(s, x(s))$ defines a $\Sigma$-measurable mapping.
\begin{remark}[Carath{\'e}odory integrand]\label{rem:appen}
We call an integrand $f$ \textit{Carath{\'e}odory} if $f$ is finite everywhere, $f(\cdot,x)$ is $\Sigma$-measurable for all $x\in X$, and  $f(s,\cdot)$ is continuous for all $s\in S$. If an integrand is Carath{\'e}odory, then it is normal (for the proof see, e.g., \cite[Lemma 4.51]{aliprantis1999infinite}). 
\end{remark}

Let $f$ be a normal integrand. We define $f^*: S \times X^*\to \R\cup \cb{\infty}$ to be the convex conjugate of $f$ in its second variable, i.e., $f^*(s,\xi)=f^*_s(\xi)$ is defined by
\begin{equation*}
  f^*_s(\xi)=\sup_{x\in X}\brac{\expect{\xi,x}_{X^*,X}-f_s(x)}.
\end{equation*}
\begin{proposition}[{\cite[Proposition 2]{rockafellar1971convex}}]
Let $f$ be a normal integrand. If for each $s\in S$, $f^*_s$ is proper (this is true if, e.g., $f\geq -c$ for some $c>0$), then $f^*$ is a convex normal integrand. If $f$ is a convex normal integrand, then $\brac{f^{*}}^*=f$.
\end{proposition}
Let $p\in (1,\infty)$ and $q = \frac{p}{p-1}$ be its dual exponent of integrability. Since $\mu$ is $\sigma$-finite, we may identify $L^p(S;X)^*$ with $L^q(S;X^*)$ (see \cite[Theorem 1.5]{showalter2013monotone}). For a given normal integrand $f$, we define an integral functional $I_f: L^p(S;X)\to \R \cup \cb{\pm \infty}$ by 
\begin{equation*}
  I_f(x)= \int_{S}f(s,x(s))d\mu(s),
\end{equation*}
if $s\mapsto f(s,x(s))$ is integrable and otherwise we set $I_f$ to be $+\infty$. Analogously, we define $I_{f^*}:L^q(S; X^*)\to \R \cup \cb{\pm \infty}$. 
\begin{proposition}[{\cite[Theorem 2]{rockafellar1971convex}}]\label{prop:appen1} Let $p \in (1,\infty)$, $q=\frac{p}{p-1}$. Let $f$ be a normal integrand. If there is an element $x\in L^p(S; X)$ such that $I_{f}(x)<\infty$, then for all $\xi \in L^q(S; X^*)$, it holds
\begin{equation}\label{eq:111:a}
I_{f^*}(\xi)= \sup_{x \in L^p(S; X)}\brac{\int_S\expect{\xi(s),x(s)}_{X^*,X}d\mu(s)- I_{f}(x)}.
\end{equation}
\end{proposition}
\begin{remark}[Measurable selection]\label{rem:901:a}
The above theorem implies a measurable selection principle for parametrized minimization problems. Namely, setting $\xi=0$ above, we have
\begin{equation*}
\int_{S} \inf_{x\in X}f(s,x)d\mu(s)= \inf_{x \in L^p(S;X)}\int_{S} f(s,x(s))d\mu(s).
\end{equation*}
In particular, if the minimum on the right-hand side is attained, the latter equality implies that there exists a $\brac{\Sigma,\mathcal{B}(X)}$-measurable function $x:S\to X$ such that $\inf_{x\in X}f(s,x)=f(s,x(s))$ $\mu$-a.e. 
\end{remark}

\section*{Acknowledgments} The authors thank Alexander Mielke and Goro Akagi for useful discussions and valuable comments. MH has been funded by Deutsche Forschungsgemeinschaft (DFG) through grant CRC 1114 ``Scaling Cascades
in Complex Systems'', Project C05 ``Effective models for materials and interfaces with many scales''. SN and MV acknowledge funding by the Deutsche Forschungsgemeinschaft (DFG, German Research Foundation) -- project number 405009441, and in the context of TU Dresden's Institutional Strategy \textit{``The Synergetic University''}.

\bibliographystyle{abbrv}
\bibliography{ref3}

\begin{thebibliography}{10}

\bibitem{aliprantis1999infinite}
C.~D. Aliprantis and K.~C. Border.
\newblock Infinite dimensional analysis: a hitchhiker’s guide.
\newblock {\em Stud. Econom. Theory}, 4, 1999.

\bibitem{allaire1992homogenization}
G.~Allaire.
\newblock Homogenization and two-scale convergence.
\newblock {\em SIAM J. Math. Anal.}, 23(6):1482--1518, 1992.

\bibitem{ambrosio2008gradient}
L.~Ambrosio, N.~Gigli, and G.~Savar{\'e}.
\newblock {\em Gradient flows: in metric spaces and in the space of probability
  measures}.
\newblock Springer Science \& Business Media, 2008.

\bibitem{andrews1998stochastic}
K.~T. Andrews and S.~Wright.
\newblock Stochastic homogenization of elliptic boundary-value problems with
  {$L^{p}$}-data.
\newblock {\em Asymptot. Anal.}, 17(3):165--184, 1998.

\bibitem{attouch1978convergence}
H.~Attouch.
\newblock Convergence de fonctionnelles convexes.
\newblock In {\em Journ{\'e}es d’Analyse Non Lin{\'e}aire}, pages 1--40.
  Springer, 1978.

\bibitem{attouch1984variational}
H.~Attouch.
\newblock {\em Variational convergence for functions and operators}, volume~1.
\newblock Pitman Advanced Publishing Program, 1984.

\bibitem{barbu2010nonlinear}
V.~Barbu.
\newblock {\em Nonlinear differential equations of monotone types in Banach
  spaces}.
\newblock Springer Science \& Business Media, 2010.

\bibitem{bogachev2007measure}
V.~I. Bogachev.
\newblock {\em Measure theory}, volume~1.
\newblock Springer Science \& Business Media, 2007.

\bibitem{bourgeat1994stochastic}
A.~Bourgeat, A.~Mikeli{\'c}, and S.~Wright.
\newblock Stochastic two-scale convergence in the mean and applications.
\newblock {\em J. Reine Angew. Math.}, 456(1):19--51, 1994.

\bibitem{brezis1973ope}
H.~Br{\'e}zis.
\newblock {\em Operateurs maximaux monotones et semi-groupes de contractions
  dans les espaces de {H}ilbert}, volume~5.
\newblock Elsevier, 1973.

\bibitem{brezis2010functional}
H.~Br{\'e}zis.
\newblock {\em Functional analysis, {S}obolev spaces and partial differential
  equations}.
\newblock Springer Science \& Business Media, 2011.

\bibitem{cioranescu2004homogenization}
D.~Cioranescu, A.~Damlamian, and R.~De~Arcangelis.
\newblock Homogenization of nonlinear integrals via the periodic unfolding
  method.
\newblock {\em C. R. Math. Acad. Sci. Paris}, 339(1):77--82, 2004.

\bibitem{cioranescu2002periodic}
D.~Cioranescu, A.~Damlamian, and G.~Griso.
\newblock Periodic unfolding and homogenization.
\newblock {\em C. R. Math. Acad. Sci. Paris}, 335(1):99--104, 2002.

\bibitem{cioranescu2008periodic}
D.~Cioranescu, A.~Damlamian, and G.~Griso.
\newblock The periodic unfolding method in homogenization.
\newblock {\em SIAM J. Math. Anal.}, 40(4):1585--1620, 2008.

\bibitem{daneri2010lecture}
S.~Daneri and G.~Savar{\'e}.
\newblock Lecture notes on gradient flows and optimal transport.
\newblock {\em arXiv preprint arXiv:1009.3737}, 2010.

\bibitem{delarue2009stochastic}
F.~Delarue and R.~Rhodes.
\newblock Stochastic homogenization of quasilinear pdes with a spatial
  degeneracy.
\newblock {\em Asymptotic Analysis}, 61(2):61--90, 2009.

\bibitem{efendiev2005homogenization}
Y.~Efendiev and A.~Pankov.
\newblock Homogenization of nonlinear random parabolic operators.
\newblock {\em Advances in Differential Equations}, 10(11):1235--1260, 2005.

\bibitem{faggionato2008random}
A.~Faggionato.
\newblock Random walks and exclusion processes among random conductances on
  random infinite clusters: homogenization and hydrodynamic limit.
\newblock {\em Electron. J. Probab.}, 13:2217--2247, 2008.

\bibitem{fatima2012unfolding}
T.~Fatima, A.~Muntean, and M.~Ptashnyk.
\newblock Unfolding-based corrector estimates for a reaction--diffusion system
  predicting concrete corrosion.
\newblock {\em Appl. Anal.}, 91(6):1129--1154, 2012.

\bibitem{griso2004error}
G.~Griso.
\newblock Error estimate and unfolding for periodic homogenization.
\newblock {\em Asymptot. Anal.}, 40(3, 4):269--286, 2004.

\bibitem{hanke2017phase}
H.~Hanke and D.~Knees.
\newblock A phase-field damage model based on evolving microstructure.
\newblock {\em Asymptot. Anal.}, 101(3):149--180, 2017.

\bibitem{heida2011extension}
M.~Heida.
\newblock An extension of the stochastic two-scale convergence method and
  application.
\newblock {\em Asymptot. Anal.}, 72(1-2):1--30, 2011.

\bibitem{heida2012stochastic}
M.~Heida.
\newblock Stochastic homogenization of heat transfer in polycrystals with
  nonlinear contact conductivities.
\newblock {\em Applicable Analysis}, 91(7):1243--1264, 2012.

\bibitem{jikov2012homogenization}
V.~V. Jikov, S.~M. Kozlov, and O.~A. Oleinik.
\newblock {\em Homogenization of differential operators and integral
  functionals}.
\newblock Springer Science \& Business Media, 2012.

\bibitem{kruger2003frechet}
A.~Y. Kruger.
\newblock On {F}r{\'e}chet subdifferentials.
\newblock {\em J. Math. Sci.}, 116(3):3325--3358, 2003.

\bibitem{liero2015homogenization}
M.~Liero and S.~Reichelt.
\newblock Homogenization of {C}ahn--{H}illiard-type equations via evolutionary
  {$\Gamma$}-convergence.
\newblock {\em Nonlinear Differ. Equat. Appl.}, 25(1):6, 2018.

\bibitem{lukkassen2002two}
D.~Lukkassen, G.~Nguetseng, and P.~Wall.
\newblock Two-scale convergence.
\newblock {\em Int. J. Pure Appl. Math.}, 2(1):35--86, 2002.

\bibitem{mathieu2007quenched}
P.~Mathieu and A.~Piatnitski.
\newblock Quenched invariance principles for random walks on percolation
  clusters.
\newblock In {\em Proceedings of the Royal Society of London A: Mathematical,
  Physical and Engineering Sciences}, volume 463, pages 2287--2307. The Royal
  Society, 2007.

\bibitem{mielke2014deriving}
A.~Mielke.
\newblock Deriving amplitude equations via evolutionary {$\Gamma$}-convergence.
\newblock {\em Discrete Contin. Dyn. Syst.}, 2015.

\bibitem{mielke2016evolutionary}
A.~Mielke.
\newblock On evolutionary {$\Gamma$}-convergence for gradient systems.
\newblock In {\em Macroscopic and Large Scale Phenomena: Coarse Graining, Mean
  Field Limits and Ergodicity}, pages 187--249. Springer, 2016.

\bibitem{mielke2014two}
A.~Mielke, S.~Reichelt, and M.~Thomas.
\newblock Two-scale homogenization of nonlinear reaction-diffusion systems with
  slow diffusion.
\newblock {\em Netw. Heterog. Media}, 9(2), 2014.

\bibitem{mielke2013nonsmooth}
A.~Mielke, R.~Rossi, and G.~Savar{\'e}.
\newblock Nonsmooth analysis of doubly nonlinear evolution equations.
\newblock {\em Calc. Var. Partial Differential Equations}, 46(1-2):253--310,
  2013.

\bibitem{mielke2007two}
A.~Mielke and A.~M. Timofte.
\newblock Two-scale homogenization for evolutionary variational inequalities
  via the energetic formulation.
\newblock {\em SIAM J. Math. Anal.}, 39(2):642--668, 2007.

\bibitem{neukamm2010homogenization}
S.~Neukamm.
\newblock {\em Homogenization, linearization and dimension reduction in
  elasticity with variational methods}.
\newblock PhD thesis, Technische Universit{\"a}t M{\"u}nchen, 2010.

\bibitem{neukamm2017stochastic}
S.~Neukamm and M.~Varga.
\newblock Stochastic unfolding and homogenization of spring network models.
\newblock {\em Multiscale Model. Simul.}, 16(2):857--899, 2018.

\bibitem{nvw2019}
S.~Neukamm, M.~Varga, and M.~Waurick.
\newblock Two-scale homogenization of abstract linear time-dependent {PDE}s.
\newblock {\em in preparation}.

\bibitem{neuss2007effective}
M.~Neuss-Radu and W.~J{\"a}ger.
\newblock Effective transmission conditions for reaction-diffusion processes in
  domains separated by an interface.
\newblock {\em SIAM J. Math. Anal.}, 39(3):687--720, 2007.

\bibitem{nguetseng1989general}
G.~Nguetseng.
\newblock A general convergence result for a functional related to the theory
  of homogenization.
\newblock {\em SIAM J. Math. Anal.}, 20(3):608--623, 1989.

\bibitem{Papanicolaou1979}
G.~C. Papanicolaou and S.~R.~S. Varadhan.
\newblock Boundary value problems with rapidly oscillating random coefficients.
\newblock In {\em Random fields, {V}ol. {I}, {II} ({E}sztergom, 1979)},
  volume~27 of {\em Colloq. Math. Soc. J\'anos Bolyai}, pages 835--873.
  North-Holland, Amsterdam-New York, 1981.

\bibitem{rockafellar1971convex}
R.~T. Rockafellar.
\newblock Convex integral functionals and duality.
\newblock In {\em Contributions to nonlinear functional analysis}, pages
  215--236. Elsevier, 1971.

\bibitem{rossi2006gradient}
R.~Rossi and G.~Savar{\'e}.
\newblock Gradient flows of non convex functionals in {H}ilbert spaces and
  applications.
\newblock {\em ESAIM Control Optim. Calc. Var.}, 12(3):564--614, 2006.

\bibitem{roubicek2013nonlinear}
T.~Roub{\'\i}{\v{c}}ek.
\newblock {\em Nonlinear partial differential equations with applications},
  volume 153.
\newblock Springer Science \& Business Media, 2013.

\bibitem{sandier2004gamma}
E.~Sandier and S.~Serfaty.
\newblock Gamma-convergence of gradient flows with applications to
  {G}inzburg-{L}andau.
\newblock {\em Comm. Pure Appl. Math.}, 57(12):1627--1672, 2004.

\bibitem{sango2011stochastic}
M.~Sango and J.~L. Woukeng.
\newblock Stochastic sigma-convergence and applications.
\newblock {\em arXiv preprint arXiv:1106.0409}, 2011.

\bibitem{serfaty2011gamma}
S.~Serfaty.
\newblock Gamma-convergence of gradient flows on {H}ilbert and metric spaces
  and applications.
\newblock {\em Discrete Contin. Dyn. Syst}, 31(4):1427--1451, 2011.

\bibitem{showalter2013monotone}
R.~E. Showalter.
\newblock {\em Monotone operators in Banach space and nonlinear partial
  differential equations}, volume~49.
\newblock American Mathematical Soc., 2013.

\bibitem{stefanelli2008brezis}
U.~Stefanelli.
\newblock The {B}r{\'e}zis--{E}keland principle for doubly nonlinear equations.
\newblock {\em SIAM J. Control Optim.}, 47(3):1615--1642, 2008.

\bibitem{varga2019stochastic}
M.~Varga.
\newblock {\em Stochastic unfolding and homogenization of evolutionary gradient
  systems}.
\newblock PhD thesis, TU Dresden, in preparation.

\bibitem{Visintin2006}
A.~Visintin.
\newblock Towards a two-scale calculus.
\newblock {\em ESAIM Control Optim. Calc. Var.}, 12(3):371--397, 2006.

\bibitem{woukeng2010periodic}
J.~L. Woukeng.
\newblock Periodic homogenization of nonlinear non-monotone parabolic operators
  with three time scales.
\newblock {\em Ann. Mat. Pura Appl.}, 189(3):357--379, 2010.

\bibitem{zeidler2013nonlinear}
E.~Zeidler.
\newblock {\em Nonlinear functional analysis and its applications: III:
  variational methods and optimization}.
\newblock Springer Science \& Business Media, 2013.

\bibitem{zhikov1982averaging}
V.~V. Zhikov, S.~M. Kozlov, and O.~A. Oleinik.
\newblock Averaging of parabolic operators.
\newblock {\em Trudy Moskovskogo Matematicheskogo Obshchestva}, 45:182--236,
  1982.

\bibitem{zhikov2006homogenization}
V.~V. Zhikov and A.~Pyatnitskii.
\newblock Homogenization of random singular structures and random measures.
\newblock {\em Izv. Math.}, 70(1):19--67, 2006.

\end{thebibliography}

\end{document}